\documentclass[11pt]{article}%
\usepackage{amsfonts}
\usepackage{color}
\usepackage{amsmath}
\usepackage{fullpage}
\usepackage{amssymb}
\usepackage{graphicx}
\usepackage{microtype}
\let\originalthebibliography\thebibliography
\renewcommand{\thebibliography}[1]{%
  \originalthebibliography{#1}\setlength{\itemsep}{0pt}%
  \setlength{\parsep}{0pt}}
\usepackage{blindtext}
\usepackage{titlesec}
\usepackage{authblk}
\usepackage[hidelinks]{hyperref}%

\newif\ifrevised

\revisedtrue

\providecommand{\U}[1]{\protect\rule{.1in}{.1in}}
\newtheorem{theorem}{Theorem}

\newtheorem{corollary}[theorem]{Corollary}
\newtheorem{assumption}[theorem]{Assumption}

\newtheorem{definition}[theorem]{Definition}
\newtheorem{example}[theorem]{Example}

\newtheorem{lemma}[theorem]{Lemma}

\newtheorem{proposition}[theorem]{Proposition}
\newtheorem{remark}[theorem]{Remark}

\newenvironment{proof}[1][Proof]{\noindent\textbf{#1.} }{\ \rule{0.5em}{0.5em}}

\newcommand{\R}{\mathbb{R}^d}

\newcommand{\norm}[1]{\left\| #1 \right\|}
\newcommand{\normsq}[1]{\left\| #1 \right\|^2}
\newcommand{\inner}[2]{\left< #1 , #2 \right>}

\newcommand{\Ep}[1]{\mathbb{E}\left[#1\right]}

\let\originalleft\left
\let\originalright\right
\renewcommand{\left}{\mathopen{}\mathclose\bgroup\originalleft}
\renewcommand{\right}{\aftergroup\egroup\originalright}

\newcommand{\E}{\mathbb E}
\newcommand{\Pp}{\mathbb P}
\newcommand{\op}{\mathrm{op}}
\newcommand{\loss}{\mathcal L}
\newcommand{\obj}{\ensuremath{\mathcal{L}}}

\renewcommand{\l}{\ell}
\newcommand{\dist}[1]{\operatorname{\bold{D}}(#1,\mathcal{W}^*)}

\newcommand{\FundingLogos}{
  \raisebox{0pt}{\includegraphics[height=1.5cm]{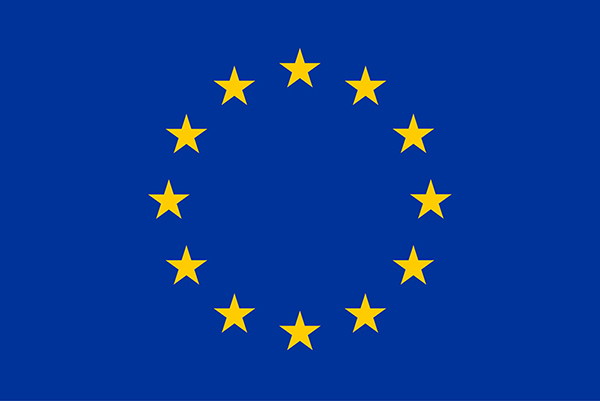}}%
  \hspace{1em}%
  \raisebox{0pt}{\includegraphics[height=1.5cm]{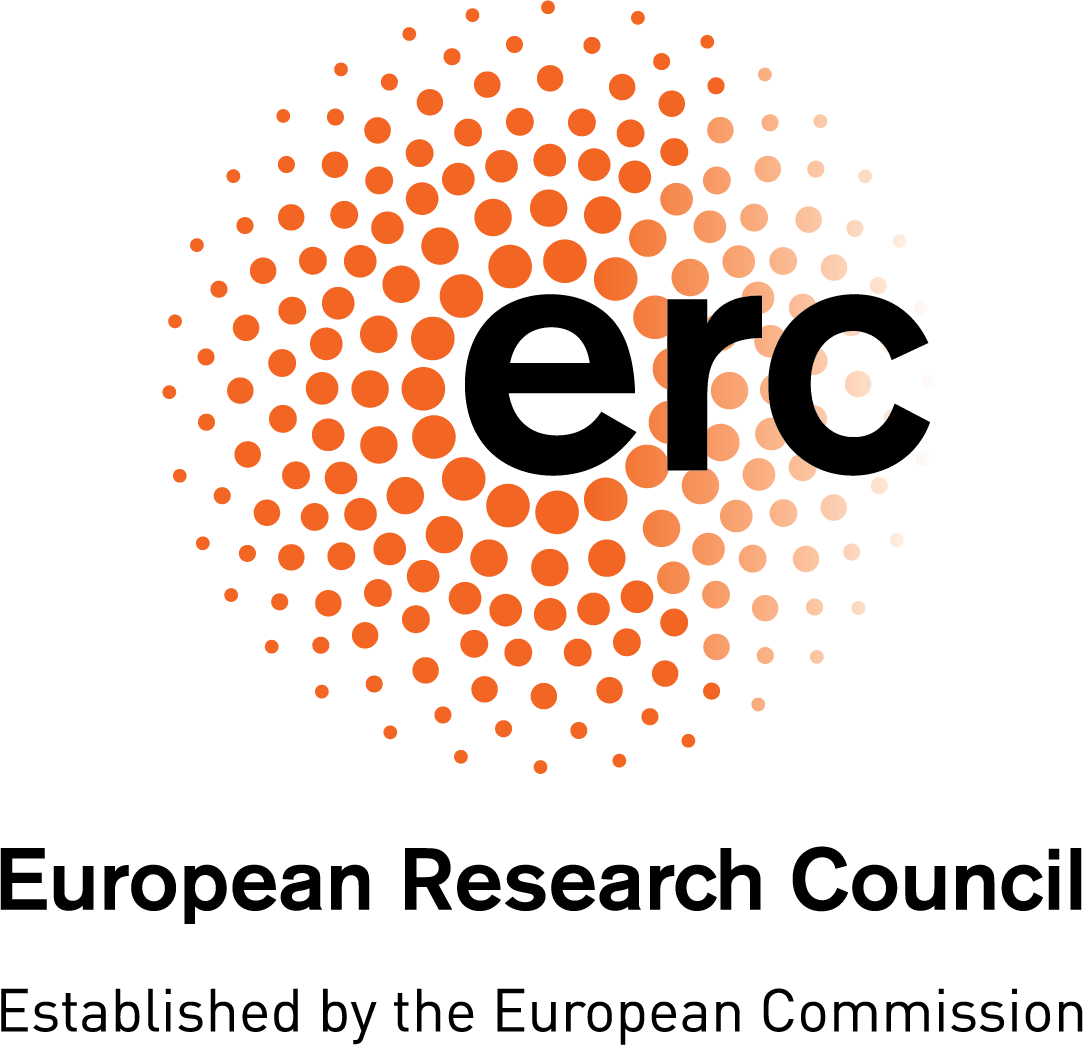}}%
}

\usepackage{tcolorbox}
\tcbset{colback=white,boxrule=0.5pt,arc=2pt,left=4pt,right=4pt,top=4pt,bottom=4pt}

\title{Large-Time Analysis of the Langevin Dynamics for Energies Fulfilling Polyak-{\L}ojasiewicz Conditions}

\author[1,2,3]{Massimo Fornasier}
\author[1,2,3]{Lukang Sun}
\author[4,5,6]{Rachel Ward}
\affil[1]{Department of Mathematics, Technical University of Munich, Garching bei M\"unchen, Germany.}
\affil[2]{Munich Center for Machine Learning, Munich, Germany}
\affil[3]{Munich Data Science Institute, Technical University of Munich, Garching bei M\"unchen, Germany.}
\affil[4]{Department of Mathematics, University of Texas, Austin, U.S.A.}
\affil[5]{Oden Institute for Computational Engineering and Sciences, University of Texas, Austin, U.S.A.}

\date{July 14, 2025}

\begin{document}

\maketitle
{
\begin{abstract}
In this work, we take a step towards understanding overdamped Langevin dynamics for the minimization of a  general class of objective functions $\obj$. We establish well-posedness and regularity of the law $\rho_t$ of the process through novel a priori estimates,  and, very importantly, we characterize the large-time behavior of $\rho_t$ under truly minimal assumptions on $\obj$.
In the case of integrable Gibbs density, the law converges to the normalized Gibbs measure. In the non-integrable case, we prove that the law diffuses. The rate of convergence, { in terms of relative Dirichlet energy with respect to the Gibbs measure}, is  $\mathcal O(1/(\sigma t))$, where $\sigma$ is the diffusion parameter.
Under  a Polyak–{\L}ojasiewicz (PL) condition on $\obj$, we also derive sharp exponential contractivity results toward the set of global minimizers. Combining these results 
we provide the first systematic convergence analysis of Langevin dynamics under PL conditions in non-integrable Gibbs settings: a first phase of exponential in time contraction   toward the set of minimizers and then a large-time exploration over it with rate $\mathcal O(1/(\sigma t))$ {in terms of the relative Dirichlet energy}.
\end{abstract}
}

\section{Introduction}

Stochastic Gradient--based algorithms are the workhorses of modern machine learning. Recall the general set-up: the aim is to minimize a loss function of the form $\mathcal{L}(w) = \frac{1}{N} \sum_{i=1}^N \ell_i(w)$, and the Stochastic Gradient update rule is 
\begin{align}
w^{(k+1)} &= w^{(k)} - \eta_k \nabla {\cal L}(w^{(k)}) + \eta_k \xi^{(k)},
\end{align}
where $\eta_k > 0$ is the step-size at iteration $k$, and $\xi^{(k)}$ is a stochastic noise vector, independent of previous $\xi^{(h)}$ $h < k$, but possibly depending on $w^{(k)}$.   One source of noise, which is almost unavoidable in large-scale machine learning applications, is due to applying mini-batch gradient descent en leiu of full gradient descent, in which case $\xi^{(k)}=  \frac{1}{|B_k|} \sum_{i\in B_k} \nabla \ell_i(w^{(k)}) - \nabla {\cal L}(w^{(k)})$, where $B_k \subset [N]$ are the indices in the mini-batch used at iteration $k$, modeled as drawn independently of previous mini-batches.  A second or alternative source of noise is injected Gaussian noise $\xi^{(k)}$ with bounded variance $\mathbb{E} \| \xi^{(k)} \|^2 \leq M$ or affine variance $\mathbb{E} \| \xi^{(k)} \|^2 \leq M_1 + M_2 \| \nabla {\cal L}(w^{(k)})\|^2 $.  Gaussian noise can be purposefully injected to render the algorithm privacy-preserving (so-called `DP SGD') \cite{abadi18}, to enable exploration of possibly non-convex loss landscapes \cite{misha1,cooper, misha2,misha3} and because doing so can cause an implicit regularization effect that often improves generalization \cite{arora22,neyshabur2015,razin2020implicit}.  
Focusing on the Gaussian noise setting, when the variance term is bounded $\mathbb{E} \| \xi^{(k)}\|^2 \leq M$, or even increases inversely with a decreasing stepsize $\eta_k$,  it makes more sense to consider the Stochastic Gradient update rule $w^{(k+1)} = w^{(k)} - \eta_k (\nabla {\cal L}(w^{(k)}) + \xi^{(k)})$ as acting on a random variable $W^{(k)}$ rather than a point $w^{(k)}$, and to analyze convergence in terms of the law of $W^{(k)}$ toward an equilibrium distribution, instead of analyzing convergence of a point estimate $w^{(k)}$ to a critical point $w^{*}$ (or better to a global minimizer) of the loss function.
{Let us recall  that the problem of ``sampling" (drawing samples) from a high dimensional probability distribution is well-known. It is intrinsically affected by the curse of dimensionality because of the measure concentration phenomenon.}  {In view of the convergence in law,} one would hope that convergence results from the ``optimization" setting find natural analogues in the ``sampling" setting, and vice versa. 

\subsection{The strongly convex regime}
{In the most classical regime, the translation between optimization and sampling results is seamless.  The most classical assumption in gradient descent convergence theory is that the loss function ${\mathcal L}$ is $\mu$-\emph{strongly convex} and $L$-Lipschitz smooth, in which case a direct proof shows that for a range of constant step-size $\eta_k = \eta$,  $w^{(k)}$ converges to the unique global minimizer of ${\cal L}$ at a linear rate proportional to the condition number $L/\mu$. When noise with bounded  variance $\mathbb{E} \| \xi^{(k)} \|_2^2 \leq M$ is added, the expected convergence  $\mathbb{E}[\mathcal{L}(w^{(k)}) - \mathcal{L}(w^{*})]$ is again  linear with rate $L/\mu$, {however, the iterates only approach \emph{a neighborhood of the global minimizer whose radius is on the order of $M\eta \frac{L}{\mu}$}, see \cite[Theorem 4.6]{bottou2018optimization} for details.}}


This suggests that in the strongly convex regime, and in the limit of small constant step-size $\eta$ where the discrete-time Stochastic Gradient update converges to a continuous-time  \emph{stochastic differential equation}, the dynamics should be such that the initial probability distribution of points is pushed towards a distribution centered at the global minimizer $w^{*}$ and having a variance proportional to $cM,$ where $c = \frac{L}{\mu} \eta \in (0,1)$. And indeed, this is the case.  Consider the formal limit of this process, the Langevin dynamics:
  \begin{equation}\label{eq:langevin}
    d w_t=-\nabla\obj(w_t)dt+\sqrt{2\sigma}dB_t.
\end{equation}
where $t \to B_t$ is a $d$-dimensional Brownian motion, $\sigma>0$, and $w_{t=0}=w_0 \in \R$ is given or drawn at random according to an initial probability measure $\rho_0 \in \mathcal P(\R)$.

The Langevin dynamics \eqref{eq:langevin} converge classically to {\it near} global minimizers under the so-called \emph{log-Sobolev inequality} (LSI).  The LSI holds for $\rho_\infty$ if, for any $\lambda>0$,
\begin{equation}
\label{eq: LSI} 
\begin{array}{c}
\int_{\R} \log \left( \frac{d\rho}{d\rho_\infty} \right) d\rho(x) =: \underbrace{H(\rho | \rho_\infty)}_{\substack{\text{relative} \\ \text{entropy}}} \leq \frac{1}{2 \lambda} \underbrace{I(\rho | \rho_\infty)}_{\substack{\text{Fisher} \\ \text{information}}} := \frac{1}{2 \lambda} \int_{\R} {\norm{\nabla \log \left( \frac{d\rho}{d\rho_\infty} \right)  }^2 }d\rho(x) \\
\mbox{ for all } \rho \in \mathcal{P}(\R).
\end{array}
\end{equation}
{(As in this paper we do not actually make use of LSI, we do not introduce the relative entropy nor the Fischer information in full glory, and we refer to \cite[Section 9.2]{villani2003} for details.)}
    For the noisy evolution \eqref{eq:langevin},
    an LSI with constant \(\lambda > 0\) ensures exponential convergence of the law \(\rho_t := \operatorname{Law}(w_t)\) in terms of Wasserstein distance or relative entropy toward the normalized  Gibbs density \( \rho_\infty \propto \pi:=e^{-\obj/\sigma}  \). Indeed, \emph{a typical $\obj$ for which \eqref{eq: LSI} holds is a strictly convex function, or an $L^\infty$ perturbation of a strictly convex function}, see \cite{hs88}. Note that in the case $\obj(w)  = \| A w - y \|^2$ where $A$ is invertible, then the Gibbs density $\pi=e^{-\obj/\sigma}$ is a Gaussian density centered at the minimizer $w^{*}$, and this density has variance $\sigma^2$, aligning with the discrete SGD results about fast convergence within a radius of the variance. 
    
 Note that the discrete-time theory for SGD allows for a general range of noise distributions including Gaussian noise as a special case.  The continuous-time Langevin dynamics theory, while specific to Gaussian noise, nevertheless provides a finer-grained picture of how the dynamics are evolving in a distributional sense: any initial distribution converges in the Wasserstein distance towards the Gibbs density $e^{-\obj/\sigma}$.
\\

\subsection{The PL inequality regime}

The strong convexity regime is classical, and the discrete-time Stochastic Gradient Descent (SGD) theory and the continuous-time Langevin dynamics theory are well-understood and related.  However, practical implementations of SGD in training, e.g., deep neural networks, are {aiming at optimizing highly nonconvex landscapes}. A better assumption than strong convexity for overparameterized neural networks is a local version of the Polyak--Łojasiewicz (PL) inequality  \cite{misha1,misha2, misha3}.

  First, let us recall the PL inequality. When the deterministic gradient flow
  $\dot{w}_t=-\nabla\mathcal{L}(w_t)$ or the gradient descent iteration
  $w^{(k+1)}=w^{(k)}-\eta\nabla\mathcal{L}(w^{(k)})$ is analyzed under the PL inequality
   \begin{equation} \label{eq:PLg}
        \frac{1}{2} \|\nabla \mathcal{L}(w)\|^2 \geq \mu \left( \mathcal{L}(w) - \min \obj\right), \qquad \mu > 0, \quad \forall w \in \R, 
    \end{equation}
    one obtains exponential convergence to the minimizer set
    $\mathcal W^*:=\operatorname*{arg\,min}\mathcal{L}$,
    matching the fast convergence rate one achieves under strong convexity assumptions on the \emph{loss function} convergence  ${\mathcal L}(w^{(k)}) \to \min \obj$ (as opposed to the pointwise convergence $w^{(k)} \to w^{*}$ which is more difficult in the PL inequality setting as $w^{*}$ is not a necessarily a point, but possibly a set of points).  {In the Stochastic Gradient regime, fast convergence to within a radius of the minimizing set can be obtained for Stochastic Gradient Descent (SGD) of the form
   \begin{equation}\label{eq:sgd}
       w^{(k+1)}=w^{(k)}-\eta\Big(\nabla\obj(w^{(k)})+\xi^{(k)}\Big),
   \end{equation}
    where $\xi^{(k)},k=0,1,2,\ldots$ are independent random noise with mean $0$ and variance bounded by $\delta$. Under an extra $L$-smoothness assumption of $\obj$, one can derive that
    \begin{equation}
    \label{eq:eta0}
    \Ep{\obj(w^{{(k+1)}})-\min \obj}\leq (1-\mu\eta)\Ep{\obj(w^{(k)})-\min \obj}+\frac{\eta^2L\delta}{2},
    \end{equation}
    here we need $\eta\leq \frac{1}{L}$, see, e.g. \cite{garrigos2023handbook}. Note that the discretized Langevin dynamics
    $w^{(k+1)}=w^{(k)}-\eta\nabla\obj(w^{(k)})+\sqrt{2\eta}B^{(k)}$,
here $B^{(k)},k=0,1,2,\ldots$ are independently sampled from the standard normal distribution $\mathcal{N}(0,I_d)$, can also be written into the form \eqref{eq:sgd} but the variance will be bounded by $\delta=\frac{2d}{\eta}$.}
The analogy between the ``strongly convex setting" and the ``PL inequality setting" breaks down at this point:  
{in the continuous-in-time limit $\eta \to 0$, the bound in equation \eqref{eq:eta0} loses meaning, hence, until now, there remains {no comprehensive theory} for the Langevin dynamics \eqref{eq:langevin} under PL inequality. }

Our  contribution is to close this gap, and provide the first analysis of Langevin dynamics 
{in greater generality} under the PL inequality, showing rigorously that under the PL condition, the Langevin dynamics decompose into two phases: a fast convergence to the set of minimizers, followed by a slower diffusion along the set of minimizers.  {The analysis of this second phase does not actually require the PL condition and holds under more general assumptions and it is a relevant result of this paper of independent interest. We reiterate that this is far outside the scope of current sufficient conditions, such as  $\log-$Sobolev inequality or Poincaré inequality, see, e.g., \cite{pmlr-v65-raginsky17a}.  Indeed}, our results give further theoretical framework for recent results along these lines: In recent work, e.g., \cite{arora22,shalova2024}, the ability of Stochastic Gradient Descent to explore the set of global minimizers once the set has been reached was modeled and studied. {The practical significance lies in the fact that this behavior allows for the identification of multiple quasi-optimal solutions, some of which may generalize better after deep learning training.} In the aforementioned studies, the dynamics of Stochastic Gradient Descent is typically divided into two phases. The first phase concerns convergence to the set of global minimizers, while the second phase describes the random, oscillatory drift of the iterates along this set, which is often assumed to form a high-dimensional manifold. 
{As studied in \cite{arora22,shalova2024}, when Stochastic Gradient Descent reaches zero loss in over parameterized models, it does not truly halt; instead, it enters a phase of {\it implicit regularization} known as stochastic drift. Even though the average gradient is zero, the remaining noise from mini-batch sampling pushes the parameters along the manifold of global minimizers. Mathematically, this movement acts as a second-order optimizer that favors flat minima over sharp ones, effectively minimizing the {\it sharpness}, i.e., the Hessian trace $\operatorname{tr}(\nabla^2 \obj(x))$ of the loss. This ``post-convergence" evolution is crucial because it reduces the complexity of the solution and improves generalization even after the training error has vanished. Technically in the vicinity of the set of minimizers the dynamics is described by a second order dynamics involving $\nabla^2 f$  (as $\nabla f(w)=0$ on $\mathcal W^*$).
A comprehensive understanding of the exploratory behavior of the Langevin dynamics after convergence to the set of global optima remains an open question, one that we contribute to further explore in this paper. Yet, we do not follow a  local second order analysis around the set  minimizers as in \cite{arora22,shalova2024} to emphasize the implicit regulariazion, but rather prove the large time diffusion of the dynamics over the manifold of minimizers. Our techniques rest in the analysis of the Fokker-Plank equation governing the law of the dynamics.}\\

{We may also mention other recent results addressing the large time behavior of the Langevin dynamics with and without PL conditions, but they are all in significantly more restrictive settings. For example, in \cite{woj24} one large-time asymptotics result, namely Theorem 2.9, is presented for a modified version of the Langevin dynamics, where the variance is time dependent and dominated by the objective function. The result is limited to objective functions $\obj$ that grow quadratically with a lower bound
$$
\obj(w) \geq \lambda (1+\|w\|)^2.
$$
This implies in particular that $\pi = e^{-\obj/\sigma}$ is integrable. Moreover, this condition also excludes that the set of minimizers is unbounded.
In the more recent work \cite{gong2025poincare}  the authors show that a {\it local} PL condition implies the Poincaré inequality, which in turn yields exponential convergence. Yet, their model of objective function does not possess saddle points, the set of minimizers must be contained in a compact set, and the Gibbs density $\pi=e^{-\mathcal L/\sigma}$ is integrable; therefore the model represents a significantly more restrictive subclass of the objective functions addressed in our present paper. Moreover, we need to stress once again that the large time convergence in our analysis (what we called the second phase of the dynamics above) is {\it not} dependent on the PL condition. Moreover, one can relate the LSI and PL settings:
For the case where the density $\rho_\infty=\pi=e^{-\obj/\sigma}$ is integrable and fulfills the LSI with constant $\lambda = \lambda_\sigma$ and $\mathcal L$ fulfills a PL condition with constant $\mu$, the work \cite{chewi2024} provides a relationship between the PL constant $\mu$ and the LSI constant $\lambda_\sigma$ for $\sigma \to 0$, namely $\mu = \lim_{\sigma \to 0} \frac{1}{\sigma \lambda_\sigma}$. Yet, also in this case, the setting is more restrictive than the one we are considering in the present paper. Taken together, these recent references support the conclusion that our work provides, to the best of our knowledge, the most general analysis of Langevin dynamics currently available. Let us now present our contribution in details.}

\subsection{Our Contribution}

This paper addresses the convergence  of the Langevin dynamics to global minimizers under  Polyak--Łojasiewicz conditions, without necessarily assuming integrability of \( e^{-\obj/\sigma} \), and the large time exploratory behavior of the dynamics over the set of minimizers. More specifically, our main contributions are:

\begin{enumerate}
    \item \textbf{Well-posedness and regularity.} For $w_t$ solution of \eqref{eq:langevin}
    We revisit results of global existence and uniqueness of the law \( \rho_t =\operatorname{Law}(w_t) \) as solution of the Fokker-Planck equation
    \begin{equation}\label{eq:FP0}
\partial_t \rho_t(w) = \operatorname{div}(\nabla \obj(w) \rho_t(w)) + \sigma \Delta \rho_t(w).
\end{equation}
    In particular, we contribute with new a priori estimates for its regularity, under assumptions of regularity on $\obj$. 
    
    {The existing literature on Fokker–Planck equations mainly addresses the existence, uniqueness, and regularity of solutions under very general conditions. Consequently, {explicit relative Dirichlet-energy decay estimates}, like those established in the present work, are generally not available for such generalized Fokker–Planck equations; see \cite{bogachev2022fokker} and references therein. When $\mathcal{L}$ is constant, the estimates obtained here can be recovered using the heat kernel representation of the heat equation solution. In contrast, for a general $\mathcal{L}$, the heat kernel is unavailable, yet our approach still yields these estimates. Another approach to studying the Fokker–Planck equation is via $\Gamma$-calculus and functional inequalities, such as the log-Sobolev and Poincar\'e inequalities, to analyze the long-time behavior of diffusion semigroups; see \cite{bakry2014analysis} and references therein. However, this approach relies on such functional inequalities, so the results for general smooth $\mathcal{L}$ presented here do not appear in the existing literature along this line, as far as we know.}
  \item \textbf{Large-time behavior of the law.} \\
    We describe the precise asymptotic behavior of \( \rho_t \), governed by the integrability of \( \pi=e^{-\obj/\sigma} \):
    \begin{itemize}
        \item \emph{Integrable case:} If \( \int_{\mathbb{R}^d} e^{-\obj(w)/\sigma} \, dw < \infty \), then \( \rho_t \) converges to the normalized Gibbs measure $\pi(w)dw$.
        \item \emph{Non-integrable case:} If $\pi(w)$ is not integrable, then  \( \rho_t(Z) \to 0 \) for all $Z\Subset \R$.\footnote{Differently from the results in, e.g., \cite{arora22,shalova2024}, we do not perform a local analysis around $\mathcal{W}^*$. In \cite{arora22,shalova2024} the authors show that the dynamics $w_t$ exhibits an oscillating drift behavior along the smooth manifold of global minimizers $\mathcal{W}^*$, while we aim at describing the dynamics of the law of the process in its entirety-.}
    \end{itemize}
       {In both cases, the relative Dirichlet energy decays at rate $\mathcal O(1/(\sigma t))$, see \eqref{eq:slowconv}; the convergence of the law stated above is qualitative.} {Here we need to stress very much, as fundamental  contribution of this paper, that  the asymptotic  results are obtained \emph{without requiring PL or LSI conditions}.}
    \item \textbf{Sharp exponential contractivity.} We prove that, under both global and local  Polyak-{\L}ojasiewicz conditions on $\obj$, the solution contracts exponentially at rate \( e^{-\mu t} \), causing the law to quickly concentrate on the minimizer set $\mathcal{W}^*:=\arg\min \obj$.

    \item \textbf{Two-phase dynamics.} Combining 2. and 3. we conclude that if $\obj$ obeys a global Polyak-{\L}ojasiewicz condition, then the dynamics will first concentrate on the set $\mathcal W^*$ of global minimizers and then it will have an asymptotic behavior according to the following dichotomy:
  \begin{enumerate}
    \item either $\obj$ has compact set of minimizers, and then necessarily $\pi(w)=e^{-\mathcal L(w)/\sigma}$ is  integrable, and $\rho_t$ converges to the normalized Gibbs measure;
    \item or $\obj$ has an unbounded set of minimizers, in which case —under the additional assumption that $\obj(w) - \min \obj \leq H(\operatorname{dist}(w,\mathcal W^*))$ for all $w$ and for some positive continuous function $H$—necessarily $\pi(w)=e^{-\mathcal L(w)/\sigma}$ is not integrable and $\rho_t$ diffuses everywhere over $\mathcal W^*$.\footnote{ The case where $\pi$ is integrable and the set of minimizers is unbounded is also fully captured by Theorem \ref{thm:main} and Theorem  \ref{thm:large-time-dichotomy}. In fact, also in this case, $\rho_t \to M_\pi^{-1}\pi$ in $L^2(\mathbb{R}^d,\pi)$ and $\int \norm{\nabla(\rho_t/\pi)}^2d\pi\leq C/{\sigma t}$. Yet, objective functions $\mathcal L$ such that the set of minimizers is unbounded and $\pi$ is integrable are quite peculiar cases, which may not even be so relevant for applications: the function needs to get steeper and steeper around the set of minimizers as $\|w\|\to \infty$, in particular the bound $\obj(w) - \min \obj \leq H(\operatorname{dist}(w,\mathcal W^*))$ needs to fail.}
\end{enumerate}  
\end{enumerate}

These results bridge PL conditions and Langevin dynamics, providing the first rigorous asymptotic analysis for Langevin dynamics in non-integrable Gibbs regimes. They also offer theoretical support for empirical observations that noisy gradient methods effectively explore flat minima—even when no stationary Gibbs measure exists.

{In summary, the primary results of this paper are:
\begin{theorem} \label{thm:main}
Assume that $\obj \in C^{1,1}(\R)$ and ${\mathfrak{L}\varphi:=\Delta \varphi-\frac{\inner{\nabla\obj}{\nabla\varphi}}{\sigma}}$ fulfills conditions A or B reported in formulae {\eqref{eq:a36} and \eqref{eq:a37}}. Then $\rho_t=\operatorname{Law}(w_t)$ is the unique smooth solution of \eqref{eq:FP0} and  $\varphi_t=\rho_t(w)/\pi(w)$ enjoys the following estimate:
\begin{equation}\label{eq:slowconv}
    {\int\norm{\nabla\varphi_t}^2d\pi(w)\leq \frac{\int\norm{\varphi_0}^2d\pi(w)}{2\sigma t},\quad\forall t> 0,}
\end{equation}
here $\varphi_t=\rho_t(w)/\pi(w)$. 
If further $\obj\in C^k(\mathbb{R}^d)$ for $k\geq 2$, we have the additional regularity estimates:
\begin{equation}
    { \int\norm{\mathfrak{F}_k\varphi_{t}}^2d\pi(w)\leq \Big(\frac{k}{2\sigma t}\Big)^k\int\varphi_0^2d\pi(w),\quad\forall t> 0,}
\end{equation}
where
\begin{equation}
        \mathfrak{F}_k\varphi:=\begin{cases}
        \mathfrak{L}^{\frac{k}{2}}\varphi& \text{$k$ is even}\\
        \nabla\mathfrak{L}^{\frac{k-1}{2}}\varphi& \text{$k$ is odd}.
    \end{cases}
    \end{equation}
  Moreover, under Assumption \ref{assumption}, Assumption \ref{assumption2}, and $\l_1>\sigma\l_2$, the Langevin dynamics~\eqref{eq:langevin} will concentrate around the set of global minimizers of $\obj$ with the following explicit convergence rate
    \begin{equation}\label{eq:expconc}
   \Ep{\obj(w_t)-\obj_*} \leq \Ep{\obj(w_0)-\obj_*} e^{-(\l_1-\sigma\l_2)t} + \sigma\l_3 \frac{1- e^{-(\l_1-\sigma\l_2)t}}{\l_1-\sigma\l_2}, t \geq 0.
 \end{equation}  
 Combining the results \eqref{eq:expconc} and \eqref{eq:slowconv}, the evolution is decomposed into an exponentially fast dynamics of concentration toward the set of minimizers $\mathcal{W}^*$, followed by a slow dynamics of exploration of the set $\mathcal{W}^*$ with rate ${\mathcal O(1/(t\sigma))}$ {  in terms of the relative Dirichlet energy}.
 \end{theorem}
}

{
\begin{remark}For a compact set $Z\Subset\mathbb{R}^d$,
    the theorem does not assert a convergence rate of $\mathcal{O}(1/(\sigma t))$ for the error $\int_Z |\rho_t(w)-M_\pi^{-1}\pi(w)| \, dw$, here $M_\pi:=\int\pi(w)dw$.
    Instead, it establishes that $\rho_t$ redistributes mass proportionally to the reference measure $\pi$ relative to a specific Dirichlet-type metric, yielding the energy dissipation rate:
\begin{equation}
    \int_Z \left \| \nabla \left( \frac{\rho_t}{\pi} \right) \right \|^2 d\pi \leq \frac{C}{\sigma t}.
\end{equation}

In Section \ref{sec:asympt}, we establish the strong convergence of $\phi_t$ to $M^{-1}_\pi$ in $L^2(\mathbb R^d,\pi)$, though without an explicit rate.
\end{remark}
}


The content of the paper is organized as follows: In Section \ref{sec:converg} we show the convergence of the solution of the Langevin dynamics to the set of global minimizers under a global Polyak-{\L}ojasiewicz condition. In Section \ref{sec:locaPL} we adapt the result to allow local convergence in a ball under a local Polyak-{\L}ojasiewicz condition. This local adaptation is motivated by applications in deep learning training.
Section \ref{sec:wpreg} is dedicated to a re-visitation of the well-posedness and regularity of solutions of Fokker-Planck equations \eqref{eq:FP} with novel a priori estimates. This preliminary regularity result will serve to justify the large time behavior, which is characterized in Section \ref{sec:asympt}.

\section{Convergence to minimizers under a global Polyak-{\L}ojasiewicz condition}\label{sec:converg}

\subsection{Assumptions}

Denote $\obj_*= \min \obj$, $\mathcal{W}^*:=\arg\min \obj$, and 
$\dist{w}:=\inf_{w'\in\mathcal{W}^*}\norm{w-w'}$.
\begin{assumption}\label{assumption}
    The objective function $\obj$ satisfies the following conditions for all $w \in \R$:
    \begin{align}
   &\obj(w)-\obj_*\leq \frac{1}{\l_1} \norm{\nabla\obj(w)}^2 ,\quad \l_1>0, \label{eq:ass1} \\
            &|\Delta \obj(w)|\leq \l_2\Big(\obj(w)-\obj_*\Big)+\l_3,\quad \l_2,\l_3>0, \label{eq:ass2}\\
            &\norm{\nabla\obj(w)}\leq H\Big(\obj(w)-\obj_*\Big),\quad\text{ for some non-negative continuous function } H:\mathbb{R}_+\to\mathbb{R}_+, \label{eq:ass3}
\end{align}
\end{assumption}
Additionally we need to require:
\begin{assumption}\label{assumption2}
    The Langevin dynamics~\eqref{eq:langevin} has a unique strong solution.
\end{assumption}
Sufficient conditions for \eqref{eq:langevin} to have a unique strong solution is that 
 $\nabla \obj$ is locally Lipchitz continuous and $\norm{\nabla \obj(w)}\leq C(1+\norm{w})$, see \cite[Corollary 6.3.1]{arnold1974stochastic}. Moreover, also when $\nabla \obj$ satisfies local integrability and super-linear growth conditions, \cite[Theorem 1.2]{xie2016sobolev} can again ensure the validity of  Assumption \ref{assumption2}.
 \\
\begin{remark}
An important consequence of the PL inequality is quadratic growth: for $\obj$ fulfilling the global Polyak-{\L}ojasiewicz condition, there is a non-negative function $F:\mathbb{R}_{+}\to\mathbb{R}_{+}$ with $F(r)\to\infty$ as $r\to\infty$, such that $\obj(w)-\obj_*\geq F(\dist{w})$, see ~\cite[Theorem 2 and Appendix A]{karimi2016linear}.
\end{remark}
 
A simple example of function $\obj$ fulfilling Assumption \ref{assumption} is given by 
\begin{equation}\label{eq:example}
   \obj(w) = \| A(w-w^*)\|^2  
\end{equation}for any matrix $A \in \mathbb R^{n\times d}$ with $n \ll d$. In this case the set $\mathcal{W}^*:=\arg\min \obj$ is the affine space $w^* + \operatorname{Ker}(A)$. For this $\obj$, one can choose $\l_1=4 \sigma_n(A^T)^2$ ($\sigma_n(A^T)$ is the minimal positive singular value of $A^T$), $\l_2\geq0$, and $\l_3=2 \operatorname{tr}{(A^TA)}=2 \|A\|_F^2$ for Assumption \ref{assumption} to hold. The example \eqref{eq:example} is a simple example for an \emph{overparameterized} loss function used in machine learning for modeling the training neural networks by means of Stochastic Gradient Descent. It is important to notice that such loss functions $\obj$ have affine spaces of global minimizers and the corresponding Gibbs density $e^{-\obj(w)}$ may not be integrable.

\begin{remark}
    Some comments on Assumption \ref{assumption} are in order. 
    \begin{itemize}
        \item 
   The first condition \eqref{eq:ass1} in Assumption \ref{assumption} is the Polyak-{\L}ojasiewicz (PL) inequality.  
 
\item For applications where gradient flows/descent methods are used, the PL condition is considered natural to describe convergence to global minimizers.  
The corresponding Gibbs density $\pi(w):=e^{-\frac{\obj(w)}{\sigma}} $ may not be integrable (an example is precisely given by \eqref{eq:example}).
Hence, $\pi(w)$ cannot be renormalized to probability measure and therefore does not fulfill the well-known $\log$-Sobolev inequality used to prove convergence of the Langevin dynamics to the invariant measure $e^{-\frac{\obj(w)}{\sigma}} dw$. We recall that the $\log$-Sobolev inequality is fulfilled, for instance,  by $L^\infty$ perturbations of strictly convex functions, see \cite{hs88}. This model of nonconvexity is thus on the one hand broader than the one provided by the PL inequality, but at the same time more restrictive as it requires integrability of $\pi(w)$, which fails even in simple examples such as  \eqref{eq:example}. 
\item While the PL condition is sufficient to prove convergence to a minimizer for the gradient flow dynamics, it appears to be incomplete to provide a similar result for the Langevin dynamics. The intuitive reason is the 
need for a condition to control the diffusion, especially in the case where $\pi(w)$ is not integrable, which results in a control of the Laplacian of $\obj$, as in the second condition \eqref{eq:ass2} of Assumption \ref{assumption}.
{In the contractivity results presented below, we assume that $\ell_2\sigma < \ell_1$, reflecting a controlled bound on the noise level. }
\item 
The last condition \eqref{eq:ass3} in Assumption \ref{assumption} is technically useful to show  that $\Ep{\int_0^{t}\nabla \obj(w_s)dB_s} =0$  (see the proof of Proposition \ref{prop1} below) and it is by no means very restrictive, for example,  one can choose $H(s)=C(1+s^p),\forall s\in\mathbb{R}_+$.
 \end{itemize}
\end{remark}

\subsection{Mass Concentration}
In this section we show that, for suitable parameters $\l_1,\dots,\l_3$ and $\sigma>0$ sufficiently small the dynamics \eqref{eq:langevin} does concentrate exponentially fast around  $\mathcal{W}^*:=\arg\min \obj$, no matter whether $\pi(w)$ is integrable.
\begin{lemma}\label{lem:vanish}
    Assume that $\obj$ fulfills Assumption \ref{assumption} and $\sigma>0$ is such that $\l_1>\sigma\l_2$. Let $w_t$ be a solution of \eqref{eq:langevin}, then 
    \begin{equation}
    \Ep{\int_0^{t}\nabla \obj(w_s)dB_s} =0, \quad \mbox{ for all } t>0.       
    \end{equation}
\end{lemma}
\begin{proof}
   For the proof, we need the following sufficient condition ~(by \cite[Definition 3.1.4, Theorem 3.2.1]{oksendal2013stochastic})
\begin{equation}
    \Ep{\int_0^{t}\norm{\nabla \obj(w_s)}^2ds}<\infty,
\end{equation}
which we verify as follows.  Define the stopping time $\tau_R:=\inf\{t\geq 0:\obj(w_t)-\obj_*\geq R\}$, then we have
    \begin{equation}
        \obj(w_{t\wedge\tau_R})-\obj(w_0)=-\int_0^{t\wedge\tau_R}\norm{\nabla \obj(w_s)}^2ds+\sigma\int_0^{t\wedge\tau_R}\Delta \obj(w_s)ds+\sqrt{2\sigma}\int_0^{t\wedge\tau_R}\nabla \obj(w_s)dB_s,
    \end{equation}
(for this equality, see for example \cite[Exercise 4.9]{oksendal2013stochastic}), take expectation from both side, we have
\begin{equation}\label{eq:a17}
    \Ep{\obj(w_{t\wedge\tau_R})-\obj_*}-\Ep{\obj(w_0)-\obj_*}=-\Ep{\int_0^{t\wedge\tau_R}\norm{\nabla \obj(w_s)}^2ds}+\sigma\Ep{\int_0^{t\wedge\tau_R}\Delta \obj(w_s)ds},
\end{equation}
this is due to the following fact: by \eqref{eq:ass3} and the definition of $\tau_R$, we have $\norm{\nabla \obj(w_s)}\leq H(\obj({w_s})-\obj_*)$ is bounded, for $s\in [0,t\wedge\tau_R]$, so 
\begin{equation}
    \Ep{\int_0^t\norm{\nabla\obj(w_s)}^21_{\tau_R\wedge t}(s)ds}<\infty,
\end{equation}
here $1_{\tau_R\wedge t}(s)=1$ if $s\leq \tau_R\wedge t$ and $1_{\tau_R\wedge t}(s)=0$ otherwise. Thus by \cite[Definition 3.1.4, Theorem 3.2.1]{oksendal2013stochastic}, we have
\begin{equation}
    \Ep{\int_0^{t\wedge\tau_R}\nabla \obj(w_s)dB_s}=\Ep{\int_0^{t}\nabla \obj(w_s)1_{t\wedge\tau_R}(s) dB_s}=0.
\end{equation} 
Choose $\sigma$ such that $\ell_1>\sigma\ell_2$ and by \eqref{eq:ass1}-\eqref{eq:ass2} in Assumption \ref{assumption} we have 
\begin{equation}
    \begin{aligned}
        \Ep{\obj(w_{t\wedge\tau_R})-\obj_*}&\leq \Ep{\obj(w_0)-\obj_*}-(\l_1-\sigma\l_2){\int_0^{t\wedge \tau_R}\Ep{\obj(w_s)-\obj_*}ds}+\sigma\l_3 t\\
        &\leq \Ep{\obj(w_0)-\obj_*}+\sigma\l_3 t,
    \end{aligned}
\end{equation}
thus 
\begin{equation}
   \mathbb P(\tau_R\leq t)R\leq \mathbb P(\tau_R\leq t)\Ep{\obj(w_{\tau_R})-\obj_*\mid \tau_R\leq t}\leq \Ep{\obj(w_{t\wedge\tau_R})-\obj^*}\leq \Ep{\obj(w_0)-\obj_*}+\sigma\l_3 t,
\end{equation}
the second inequality in the above is due to
\begin{equation}
    \begin{aligned}
        \Ep{\obj(w_{t\wedge\tau_R})-\obj_*}&=\mathbb P(\tau_R\leq t)\Ep{\obj(w_{\tau_R})-\obj_*\mid \tau_R\leq t}+\mathbb{P}(\tau_R>t)\Ep{\obj(w_t)-\obj_*\mid \tau_R>t}\\
        &\geq \mathbb P(\tau_R\leq t)\Ep{\obj(w_{\tau_R})-\obj_*\mid \tau_R\leq t}.
    \end{aligned}
\end{equation}
Thus 
\begin{equation}
    \lim_{R\to\infty} \mathbb P(\tau_R\leq t)=0,\quad\forall t>0,
\end{equation}
which means $\tau_R\wedge t\to t$ almost surely for $R\to \infty$, thus $\int_0^{t\wedge\tau_R}\norm{\nabla \obj(w_s)}^2ds\to \int_0^{t}\norm{\nabla \obj(w_s)}^2ds$ almost surely. Again, by equality \eqref{eq:a17} and the assumption, we can derive
\begin{equation}
    \Ep{\int_0^{t\wedge\tau_R}\norm{\nabla \obj(w_s)}^2ds}\leq \frac{{\Ep{\obj(w_0)-\obj_*+\sigma \ell_3 t\wedge\tau_R}}}{1-\frac{\ell_2}{\ell_1}\sigma}\leq \frac{\Ep{\obj(w_0)-\obj_*}+\sigma \ell_3 t}{1-\frac{\ell_2}{\ell_1}\sigma},
\end{equation}
so let $R\to\infty$, and, by Fatou's lemma, we have
\begin{equation}
    \Ep{\int_0^{t}\norm{\nabla \obj(w_s)}^2ds}\leq \frac{\Ep{\obj(w_0)-\obj_*}+\sigma\ell_3 t}{1-\frac{\ell_2}{\ell_1}\sigma}<\infty.
\end{equation}
Hence, by \cite[Definition 3.1.4, Theorem 3.2.1]{oksendal2013stochastic} we conclude that $\Ep{\int_0^{t}\nabla \obj(w_s)dB_s} =0$. 
\end{proof}

\begin{proposition}\label{prop1}
Assume that $\obj$ fulfills Assumption \ref{assumption} and $\sigma>0$ is such that $\l_1>\sigma\l_2$. Then 
\begin{equation}\label{eq:decay}
  \Ep{\obj(w_t)-\obj_*} \leq \Ep{\obj(w_0)-\obj_*} e^{-(\l_1-\sigma\l_2)t} + \sigma\l_3 \frac{1- e^{-(\l_1-\sigma\l_2)t}}{\l_1-\sigma\l_2}, 
\end{equation}
which implies
\begin{equation}\label{eq:conc}
    \Ep{\obj(w_t)-\obj_*}\leq C \frac{\sigma\l_3}{\l_1-\sigma\l_2},
\end{equation}
for $t>0$ large enough.
\end{proposition}
\begin{proof} By It\^o's formula, we have
\begin{equation}
    d \obj(w_t)=-\norm{\nabla\obj(w_t)}^2dt+\sigma\Delta \obj(w_t)dt+\sqrt{2\sigma}\nabla\obj(w_t)dB_t.
\end{equation}
We  first reformulate the latter equation in integral form 
 \begin{equation}
        \obj(w_{t})-\obj(w_0)=-\int_0^{t}\norm{\nabla \obj(w_s)}^2ds+\sigma\int_0^{t}\Delta \obj(w_s)ds+\sqrt{2\sigma}\int_0^{t}\nabla {\mathcal{L}(w_s)}dB_s.
    \end{equation}
Then we take the expectation
 \begin{equation}\label{eq:8}
       \Ep{ \obj(w_{t})-\obj(w_0)}=-\Ep{\int_0^{t}\norm{\nabla \obj(w_s)}^2ds}+\sigma\Ep{\int_0^{t}\Delta \obj(w_s)ds}+\sqrt{2\sigma}\Ep{\int_0^{t}\nabla \obj(w_s)dB_s}.
    \end{equation}
By Lemma \ref{lem:vanish} the last term $\Ep{\int_0^{t}\nabla \obj(w_s)dB_s} =0$ vanishes.
Then by differentiating in time in \eqref{eq:8} and using   Assumption \ref{assumption}, we have
\begin{equation}
    \begin{aligned}
        \frac{d}{dt}\Ep{\obj(w_t)-\obj_*}&\leq \Ep{-\norm{\nabla\obj(w_t)}^2+\sigma\Delta\obj(w_t)}\\
        &\leq -\l_1\Ep{\obj(w_t)-\obj_*}+\Ep{\sigma\Delta\obj(w_t)}\\
        &\leq -(\l_1-\sigma\l_2)\Ep{\obj(w_t)-\obj_*}+\sigma\l_3.
    \end{aligned}
\end{equation}
Then by Gr\"onwall's lemma we obtain
\begin{equation}
  \Ep{\obj(w_t)-\obj_*} \leq \Ep{\obj(w_0)-\obj_*} e^{-(\l_1-\sigma\l_2)t} + \sigma\l_3 \frac{1- e^{-(\l_1-\sigma\l_2)t}}{\l_1-\sigma\l_2}, 
\end{equation}
which, for $\sigma$ such that $\l_1>\sigma\l_2$, implies,
\begin{equation}
    \Ep{\obj(w_t)-\obj_*}\leq C \frac{\sigma\l_3}{\l_1-\sigma\l_2},
\end{equation}
 for $t>0$ large enough.
\end{proof}
\begin{remark}\label{quadr}
For instance, if $\l_1\gg 1$, and $0<\sigma \ll 1$, then the above estimates  ensure that the process concentrates around $\mathcal{W}^*$ at the noise level $\sigma>0$. 
For the model \eqref{eq:example} the estimates \eqref{eq:decay}-\eqref{eq:conc} give
$$
\Ep{\obj(w_t)-\obj_*} \leq \Ep{\obj(w_0)-\obj_*} e^{-4\sigma_n(A^T)^2t} + \sigma \|A\|_F^2 \frac{1- e^{-4\sigma_n(A^T)^2t}}{4\sigma_n(A^T)^2}, 
$$
and
$$
\Ep{\obj(w_t)-\obj_*}\leq C  \frac{\|A\|_F^2}{4 \sigma_n(A^T)^2}\sigma,
$$
for $t\gg0$ large enough. These estimates are sharp in the sense that there are objective functions $\obj$ that equate the estimates: assume that $A=[ \quad I_n \quad | \quad 0 \quad] \in \mathbb R^{n \times d}$, then all the estimates in the proof of Proposition \ref{prop1} are actually identities. In the more general case, one can consider without loss of generality $A=[ \quad \Sigma \quad | \quad 0 \quad] \in \mathbb R^{n \times d}$, where $\Sigma =\operatorname{diag}(\sigma_1,\dots, \sigma_n)\in \mathbb R^{n \times n}$ is a diagonal matrix with positive diagonal values $\sigma_i>0$, then from \eqref{eq:8} one can easily derive the lower bound estimate
$$
\frac{d}{dt} \Ep{\obj(w_t)-\obj_*} \geq -4 \sigma_1^2 \Ep{\obj(w_t)-\obj_*} + 2 \sigma \sum_{i=1}^n \sigma_i^2,
$$
yielding
$$
\Ep{\obj(w_t)-\obj_*} \geq \Ep{\obj(w_0)-\obj_*} e^{-4 \sigma_1^2 t} + 2 \sigma \sum_{i=1}^n \sigma_i^2 \left( \frac{1 - e^{-4 \sigma_1^2 t}}{4 \sigma_1^2} \right).
$$
Notice that $\sum_{i=1}^n \sigma_i^2=\|A\|_F^2$. Hence, in this case,
$$
C \frac{\|A\|_F^2}{4 \sigma_1(A^T)^2}\sigma \leq \Ep{\obj(w_t)-\obj_*}\leq C  \frac{\|A\|_F^2}{4 \sigma_n(A^T)^2}\sigma,
$$
for $t\gg0$ large enough for a suitable constant $C>0$. One can also notice that all the constants and relevant quantities in the estimates do depend on the dimension $n$, but not on the dimension $d \geq n$.
\end{remark}

	{
		Under the PL condition, the objective satisfies the inverse growth condition $\obj(w)-\obj_*\geq \frac{\ell_1}{4}\dist{w}^2$ ~( see \cite[Theorem 2 and Appendix A]{karimi2016linear}). This inverse growth condition and Markov inequality allow to
		 derive the following result of concentration around minimizers.
		
		\begin{proposition}[Concentration of the process]\label{prop:concentration}
			Under Assumption \ref{assumption}, if $\ell_1 > \sigma \ell_2$, then for any 
			\[
			t \geq \max\Big\{-\frac{\log\Big(\frac{\sigma \ell_3}{(\ell_1 - \sigma \ell_2) \, \mathbb{E}[\mathcal{L}(w_0)-\mathcal{L}_*]}\Big)}{\ell_1 - \sigma \ell_2},0 \Big\},
			\]
			we have that  for any $\varepsilon>0$,
			\begin{equation}
				\rho_t\big(\mathbf{D}(w,\mathcal{W}^*) \leq \varepsilon\big) \geq 1 - \frac{8 \sigma \ell_3}{(\ell_1 - \sigma \ell_2) \ell_1 \varepsilon^2}.
			\end{equation}
		\end{proposition}
		
		\begin{proof}
			From estimate \eqref{eq:decay} and the condition on $t$, we have
			\[
			\mathbb{E}[\mathcal{L}(w_t) - \mathcal{L}_*] \leq \frac{2 \sigma \ell_3}{\ell_1 - \sigma \ell_2}.
			\]
			
			Applying the inverse growth condition gives
			\[
			\frac{\ell_1}{4} \, \mathbb{E}[\mathbf{D}(w_t,\mathcal{W}^*)^2] \leq \mathbb{E}[\mathcal{L}(w_t) - \mathcal{L}_*] \leq \frac{2 \sigma \ell_3}{\ell_1 - \sigma \ell_2},
			\]
			so that
			\[
			\mathbb{E}[\mathbf{D}(w_t,\mathcal{W}^*)^2] \leq \frac{8 \sigma \ell_3}{(\ell_1 - \sigma \ell_2)\ell_1}.
			\]
			
			Finally, applying Markov's inequality yields
			\[
			\rho_t\big(\mathbf{D}(w,\mathcal{W}^*) \leq \varepsilon\big) \geq 1 - \frac{8 \sigma \ell_3}{(\ell_1 - \sigma \ell_2)\ell_1 \varepsilon^2},
			\]
			for any $\varepsilon>0$.
		\end{proof}

	}

{

\section{Convergence to minimizers under a local Polyak-{\L}ojasiewicz condition}\label{sec:locaPL}

The square loss used to train a neural network is generally nonconvex
and need not satisfy a global PL inequality. For sufficiently wide
networks, however, the loss satisfies a local PL inequality on a ball
around random initialization \cite{misha3}. This local property is
sufficient for the convergence analysis of mini-batch stochastic
gradient descent in \cite{misha3}. Here we use it to study Langevin
dynamics over a finite time interval.  We first introduce the
network and its training loss.

{Fix a dataset $\{(x_i,y_i)\}_{i=1}^N$ with $y_i\in\mathbb R$
and pairwise nonparallel inputs $x_i\in\mathbb R^p$, that is,
$x_i\notin\operatorname{span}\{x_j\}$ for all $i\ne j$.}
The input dimension $p$, sample size $N$, depth
$L$, dataset, and activation $\phi$ remain fixed as the common hidden
width $m$ increases. We use the normalized feedforward network of
\cite[Section~4.2, equation~(17)]{misha3}, with linear scalar output:
\begin{equation}\label{eq:network}
\begin{aligned}
\alpha^{(0)}(x)&=x,\qquad m_0=p,\quad m_l=m\quad(1\leq l\leq L),\\
\alpha^{(l)}(x)&=\phi\!\left(W^{(l)}\alpha^{(l-1)}(x)/\sqrt{m_{l-1}}\right),
 \quad 1\leq l\leq L,\\
f_m(w;x)&=v^\top\alpha^{(L)}(x)/\sqrt m.
\end{aligned}
\end{equation}
Here $W^{(1)}\in\mathbb R^{m\times p}$,
$W^{(l)}\in\mathbb R^{m\times m}$ for $2\leq l\leq L$, and
$v\in\mathbb R^m$. The parameter vector $w$ concatenates all
these entries and has dimension
${d(m)}=pm+(L-1)m^2+m$. All entries are initialized independently from
$\mathcal N(0,1)$; we write $w_{0,m}$ for the resulting vector.
In the notation of \cite[Eq.~(17)]{misha3}, this corresponds to $m_{L+1}=1$, $W^{(L+1)}=v^\top$, and the identity output activation $\sigma_{L+1}(z)=z$.
The training task is to find parameters $w$ for which
$f_m(w;x_i)\approx y_i$ for all $i$; minimizing the empirical square
loss below expresses this task.

Define the prediction map $F_m(w)=(f_m(w;x_i))_{i=1}^N$,
its Jacobian $J_m(w)=DF_m(w)\in\mathbb R^{N\times {d(m)}}$, and the
tangent kernel $K_m(w)=J_m(w)J_m(w)^\top$. Put
$\lambda_{0,m}=\lambda_{\min}(K_m(w_{0,m}))$ and
$y=(y_1,\ldots,y_N)^\top$. All derivatives are with respect to $w$. Norms on vectors are
Euclidean, and $\|\cdot\|_{\op}$ and $\|\cdot\|_{\mathrm F}$
denote the operator and Frobenius norms. For each output component $F_{m,i}(w)=f_m(w;x_i)$, write
$H_{m,i}(w)=\nabla^2F_{m,i}(w)\in\mathbb R^{{d(m)}\times {d(m)}}$.
Thus $D^2F_m(w)$ is the $N\times {d(m)}\times {d(m)}$ tensor with
$(D^2F_m(w))_{ijk}=\partial_{w_j}\partial_{w_k}F_{m,i}(w)$,
and $D^2F_m(w)[u,z]=(u^\top H_{m,i}(w)z)_{i=1}^N$.
Following \cite[p.~6]{misha3}, set
$\|D^2F_m(w)\|_{\max}
=\max_i\|H_{m,i}(w)\|_{\op}$ and
$h_m(R)=\max_i\sup_{\|w-w_{0,m}\|\leq R}
\|H_{m,i}(w)\|_{\op}$. This tensor norm therefore controls the Hessian of each
scalar prediction. The loss is
\begin{equation}\label{eq:loss}
{\loss_m}(w)=\frac1N\|F_m(w)-y\|^2{,\qquad
\loss_{m,*}:=\inf_{w\in\mathbb R^{d(m)}}\loss_m(w)}.
\end{equation}

For the sample losses $\ell_i(w)=|f_m(w;x_i)-y_i|^2$,
mini-batch SGD draws indices $i_1,\ldots,i_h$ from
$\{1,\ldots,N\}$ and updates
$w^{(k+1)}=w^{(k)}-(\Delta t/h)\sum_{j=1}^h
\nabla\ell_{i_j}(w^{(k)})$.
The analysis in \cite{misha3} keeps these iterates inside a ball
around initialization and applies a local PL inequality there, even
though a global PL inequality need not hold. The reason
wide networks admit such an estimate is that their outputs become
approximately affine in the parameters on each fixed ball: the
Hessians of the scalar predictions tend to zero as the width grows.

\subsection{Local estimates}\label{sec:local-estimates}

Let $w_{0,m}$ be the random initialization, with each component sampled independently from $\mathcal{N}(0,1)$ and let the activation function $\phi$ be real analytic, non-polynomial with bounded first and second derivatives.
{Under the nonparallel-input assumption,
\cite[Condition~3.2 and Proposition~F.1]{Du2019}, the discussion
before \cite[Theorem~4, p.~13]{misha3}, and the feature-kernel
estimates in \cite[Appendix~E]{Du2019} give a constant $\kappa>0$,
independent of $m$, such that $\lambda_{0,m}\geq\kappa$ with
probability tending to one as $m\to\infty$.}
For each $0<\eta<1$, the
initialization estimates in \cite[Proof of Theorem~8, p.~16]{misha3}
allow finite constants $M_\eta>0$ and $B_\eta\geq0$ such that
$\|J_m(w_{0,m})\|_{\op}\leq M_\eta$ and
$\|F_m(w_{0,m})-y\|\leq B_\eta$ hold together with probability at
least $1-\eta/4$ for all sufficiently large $m$; here
$\|J_m(w_{0,m})\|_{\op}^2=\|K_m(w_{0,m})\|_{\op}$.
Since $\lambda_{0,m}\geq\kappa$ also holds with probability at least
$1-\eta/4$ for all sufficiently large $m$, the union bound gives the
following inequalities with probability at least $1-\eta/2$ for all
sufficiently large $m$:
\begin{equation}\label{eq:uniform-init}
 \lambda_{0,m}\geq\kappa,\qquad
 \|J_m(w_{0,m})\|_{\op}\leq M_\eta,\qquad
 \|F_m(w_{0,m})-y\|\leq B_\eta.
\end{equation}
For the network and Gaussian initialization specified above,
$\kappa$ may depend on the dataset, the depth $L$, and the
activation $\phi$, but not on $\eta$. The constants $M_\eta$ and
$B_\eta$ may depend on these same quantities and on $\eta$.
None of $\kappa$, $M_\eta$, and $B_\eta$ depends on $m$, ${d(m)}$, and, obviously, on any radius $R>0$ of a ball around $w_{0,m}$ (simply because nothing in the model depends on $R$).

Let $0<\eta<1$ and choose $\kappa,M_\eta,B_\eta$ as above. Define
$\theta=\min\{1,\sqrt\kappa/2\}$ and $M=M_\eta+\theta$.
In particular, $0<\theta\leq\sqrt\kappa/2$.
Define
\begin{equation}\label{eq:constants}
\ell_0'=\frac{B_\eta^2}{N},\quad
\ell_1'=\frac{\kappa}{N},\quad
\ell_2'=\frac{2}{N}[M^2+\theta(B_\eta+M)],\quad
\ell_3'=\frac{8M^2}{N^2}\max\{B_\eta^2,M^2\}.
\end{equation}
The constants $\ell_0'$, $\ell_1'$, $\ell_2'$, and $\ell_3'$ are independent
of $m$, ${d(m)}$, and, again, obviously, of $R>0$ (as the model does not depend on it). Among them, $\ell_0'$, $\ell_2'$, and $\ell_3'$
may depend on $\eta$, whereas $\ell_1'$ does not.

Let $R>0$ now be fixed, for each fixed input $x_i$,
\cite[Theorem~3.2 and Appendix~F]{misha2}, with the radius dependence
made explicit in \cite[Theorem~5, (19)]{misha3}, gives, with probability
tending to one as $m\to\infty$ over the above Gaussian initialization,
$\sup_{\|w-w_{0,m}\|\leq R}\|\nabla_w^2 f_m(w;x_i)\|_{\op}
=\widetilde O(R^{3L}/\sqrt m)$~( here $\widetilde O$ retains the logarithmic factors).
Since $R$ and $N$ are fixed, the Hessian estimate and a union bound over
the $N$ inputs, together with the initialization estimate
\eqref{eq:uniform-init}, give an integer $m_H(R,\eta,N,\kappa)$ such
that, for every $m\geq m_H(R,\eta,N,\kappa)$,
\eqref{eq:uniform-init} holds with probability at least $1-\eta/2$,
and the bound
\begin{equation}\label{eq:h-small}
 h_m(R)\leq\frac{\theta}{\sqrt N(1+R)}
\end{equation}
holds with probability at least $1-\eta/2$. Consequently,
\eqref{eq:uniform-init} and \eqref{eq:h-small} hold together
with probability at least $1-\eta$ for every $m\geq m_H(R,\eta,N,\kappa)$.
Combining the feature-kernel estimate in \cite[Appendix E]{Du2019} with the Hessian estimate in \cite[Theorem 5]{misha3} and its probability bounds in \cite[Appendix F]{misha2}, one may choose a sufficient width of the following order:
\[
 m_H(R,\eta,N,\kappa)=\widetilde O\!\left(
 \max\left\{
 \frac{N(1+R)^{6L+2}}{\min\{1,\kappa/4\}},\;
 \frac{2^{O(L)}N^2}{\kappa^2}\log\frac{CLN}{\eta},\;
 \exp\!\left(C\sqrt{\log\frac{CLN}{\eta}}\right)
 \right\}\right),
\]
here $C>0$ depends on the fixed data, depth, and activation.
\begin{proposition}\label{prop:main}
{With nonparallel data and the initialization and activation specified above,
for every $\eta\in (0,1)$, $R>0$, and $m\geq m_H(R,\eta,N,\kappa)$,
both \eqref{eq:uniform-init} and \eqref{eq:h-small} hold with probability
at least $1-\eta$ over $w_{0,m}\sim\mathcal N(0,I_{{d(m)}})$.} On the same event, for every
$w\in\overline B_R(w_{0,m})$,
\begin{align}
\|\nabla{\loss_m}(w)\|^2
 &\geq \ell_1'({\loss_m}(w)-{\loss_{m,*}}),\label{eq:PL}\\
\|\nabla^2{\loss_m}(w)\|_{\op}&\leq \ell_2',\label{eq:H}\\
|\Delta{\loss_m}(w)|&\leq {d(m)}\ell_2',\label{eq:Delta}\\
\|\nabla{\loss_m}(w)\|^2
 &\leq \ell_3'(1+\|w-w_{0,m}\|^2).\label{eq:growth}
\end{align}
Here $\ell_1',\ell_2',\ell_3'$ are given by \eqref{eq:constants}.
\end{proposition}
\begin{remark}
{The right-hand side of \eqref{eq:Delta} depends on $d(m)$.
This causes no difficulty, since this bound always appears multiplied
by $\sigma$. Choosing $\sigma$ proportional to $1/d(m)$ makes the
product independent of $d(m)$, see Proposition~\ref{prop:noise} and Corollary~\ref{cor:finite-time-family} below.}
\end{remark}
\begin{proof}
Fix $m\geq m_H(R,\eta,N,\kappa)$, as shown above,
\eqref{eq:uniform-init} and \eqref{eq:h-small} hold simultaneously
with probability at least $1-\eta$. We now fix an initialization
for which both conditions hold and use them to show \eqref{eq:PL}-\eqref{eq:growth}.

\noindent
Step 1. Bound the Jacobian above and its smallest singular
value below throughout the ball.
Let $m\geq m_H(R,\eta,N,\kappa)$ and $w\in\overline B_R(w_{0,m})$.
Write $u=w-w_{0,m}$, so $\|u\|\leq R$.
By \eqref{eq:uniform-init}, $\|J_m(w_{0,m})\|_{\op}\leq M_\eta$.
Moreover, since $K_m=J_mJ_m^\top$, for every $z\in\mathbb R^N$ we have
$\|J_m(w_{0,m})^\top z\|^2
=z^\top K_m(w_{0,m})z\geq\kappa\|z\|^2$.
As in \cite[Proof of Theorem~2, p.~10]{misha3}, the fundamental
theorem of calculus along the segment $w_{0,m}+su$, $0\leq s\leq1$,
which lies in the ball, gives
$\nabla F_{m,i}(w)-\nabla F_{m,i}(w_{0,m})
=\int_0^1 H_{m,i}(w_{0,m}+su)u\,ds$.
Every row of the Jacobian difference therefore has norm at most
$h_m(R)\|u\|$. Summing the squared row norms yields
\begin{align}
\|J_m(w)-J_m(w_{0,m})\|_{\op}
&\leq\|J_m(w)-J_m(w_{0,m})\|_{\mathrm F}\nonumber\\
&\leq\sqrt N\,h_m(R)\|u\|
\leq\frac{\theta R}{1+R}\leq\theta.\label{eq:J-difference}
\end{align}
Consequently,
\begin{equation}\label{eq:J-local}
\|J_m(w)\|_{\op}\leq M_\eta+\theta=M.
\end{equation}
For every $z\in\mathbb R^N$, the reverse triangle inequality also gives
\begin{align}
\|J_m(w)^\top z\|
&\geq\|J_m(w_{0,m})^\top z\|
 -\|(J_m(w)-J_m(w_{0,m}))^\top z\|\nonumber\\
&\geq(\sqrt\kappa-\theta)\|z\|
\geq\frac{\sqrt\kappa}{2}\|z\|.\label{eq:row-lower}
\end{align}

\noindent
Step 2. Bound the residual by $B_\eta+M\|w-w_{0,m}\|$.
Put $r_m(w)=F_m(w)-y\in\mathbb R^N$.
The fundamental theorem of calculus gives
$r_m(w)=r_m(w_{0,m})+\int_0^1J_m(w_{0,m}+su)u\,ds$.
Using \eqref{eq:uniform-init} and \eqref{eq:J-local}, we obtain
\begin{equation}\label{eq:residual}
\|r_m(w)\|\leq B_\eta+M\|u\|\leq B_\eta+MR.
\end{equation}

\noindent
Step 3. Prove the local PL inequality \eqref{eq:PL}.
Differentiating \eqref{eq:loss} gives
$\nabla{\loss_m}(w)=\frac2N J_m(w)^\top r_m(w)$.
Thus \eqref{eq:row-lower} implies
\[
\|\nabla{\loss_m}(w)\|^2
\geq\frac4{N^2}\frac\kappa4\|r_m(w)\|^2
=\frac\kappa N{\loss_m}(w)
\geq\frac\kappa N({\loss_m}(w)-{\loss_{m,*}}).
\]
The last inequality uses ${\loss_{m,*}}\geq0$.
This proves \eqref{eq:PL}.

\noindent
Step 4. Prove the Hessian bound \eqref{eq:H}.
Differentiating the gradient identity from Step~3 gives
\begin{equation}\label{eq:loss-Hessian}
\nabla^2{\loss_m}(w)=\frac2N\left[
 J_m(w)^\top J_m(w)+\sum_{i=1}^N r_{m,i}(w)H_{m,i}(w)\right].
\end{equation}
For the first term, $\|J_m^\top J_m\|_{\op}=\|J_m\|_{\op}^2\leq M^2$.
For the second term, the triangle inequality and Cauchy--Schwarz give
$\|\sum_i r_{m,i}H_{m,i}\|_{\op}
\leq h_m(R)\sum_i|r_{m,i}|
\leq\sqrt N\,h_m(R)\|r_m\|$.
Therefore \eqref{eq:h-small} and \eqref{eq:residual} yield
\begin{align*}
\|\nabla^2{\loss_m}(w)\|_{\op}
&\leq\frac2N\left[M^2+\sqrt N\,h_m(R)(B_\eta+MR)\right]\\
&\leq\frac2N\left[M^2+
       \theta\frac{B_\eta+MR}{1+R}\right]\\
&\leq\frac2N\left[M^2+\theta(B_\eta+M)\right]=\ell_2'.
\end{align*}
The final inequality uses $B_\eta/(1+R)\leq B_\eta$ and
$MR/(1+R)\leq M$. This proves \eqref{eq:H}.

\noindent
Step 5. Deduce the Laplacian bound \eqref{eq:Delta}.
The loss Hessian is a symmetric ${d(m)}\times {d(m)}$ matrix, so
$|\Delta{\loss_m}(w)|=|\operatorname{tr}\nabla^2{\loss_m}(w)|
\leq\sum_{j=1}^{{d(m)}}|\lambda_j(\nabla^2{\loss_m}(w))|
\leq {d(m)}\ell_2'$.
This proves \eqref{eq:Delta}.

\noindent
Step 6. Prove the gradient-growth bound \eqref{eq:growth}.
The gradient identity from Step~3, together with
\eqref{eq:J-local} and \eqref{eq:residual}, gives
\begin{align*}
\|\nabla{\loss_m}(w)\|^2
&\leq\frac{4M^2}{N^2}(B_\eta+M\|u\|)^2\\
&\leq\frac{8M^2}{N^2}(B_\eta^2+M^2\|u\|^2)\\
&\leq\frac{8M^2}{N^2}\max\{B_\eta^2,M^2\}(1+\|u\|^2)
 =\ell_3'(1+\|u\|^2).
\end{align*}
The second line uses $(s+t)^2\leq2s^2+2t^2$.
This proves \eqref{eq:growth}.
\end{proof}

\subsection{Langevin dynamics under a local PL condition}
We now combine the local estimates of Proposition~\ref{prop:main} with an exit-time bound to control the loss of the Langevin dynamics at time $T$.
\begin{proposition}\label{prop:noise}
Fix $0<\eta<1$, $R>0$, and $m\geq m_H(R,\eta,N,\kappa)$, then
with probability at least $1-\eta$ over the Gaussian initialization
$w_{0,m}\sim\mathcal N(0,I_{d(m)})$, the conclusions of
Proposition~\ref{prop:main} hold. For each initialization in this
event, let $w_t$ be the Langevin dynamics for ${\loss_m}$, started at
$w_{0,m}$ with noise level $\sigma>0$. Define
\begin{equation}\label{eq:AC}
\begin{aligned}
 A(T)&=\ell_0' e^{-\ell_1'T}
       +\frac{\sigma {d(m)}\ell_2'}{\ell_1'}(1-e^{-\ell_1'T}),\\
 C(T)&=2T\ell_0'+2\sigma {d(m)}\ell_2'T^2+4{d(m)}\sigma T.
\end{aligned}
\end{equation}
For $T,\epsilon>0$, conditional on any such initialization,
\begin{equation}\label{eq:prob-bound}
\Pp\{{\loss_m}(w_T)-{\loss_{m,*}}\geq\epsilon\}
 \leq\frac{A(T)}{\epsilon}+\frac{C(T)}{R^2}.
\end{equation}
Consequently, if $0<\delta<1$ and $T$ satisfies
\begin{equation}\label{eq:TR}
 T\geq T_0:=
 \begin{cases}
  \displaystyle\max\left\{0,\frac1{\ell_1'}
      \log\frac{\ell_1'\ell_0'}{\sigma {d(m)}\ell_2'}\right\},
        &\ell_0'>0,\\
  0,&\ell_0'=0,
 \end{cases}
 \qquad C(T)\leq R^2\delta,
\end{equation}
then
\begin{equation}\label{eq:joint-prob}
\Pp\{{\loss_m}(w_T)-{\loss_{m,*}}<\epsilon\}
 \geq1-\frac{2\sigma {d(m)}\ell_2'}{\ell_1'\epsilon}-\delta.
\end{equation}
\end{proposition}
\begin{proof}
For the fixed initialization $w_{0,m}$ in the event, set
$\tau_R=\inf\{t\geq0:\|w_t-w_{0,m}\|\geq R\}$.
The initialization and local bounds in
Proposition~\ref{prop:main} give
${\loss_m}(w_{0,m})-{\loss_{m,*}}\leq{\loss_m}(w_{0,m})
=\|F_m(w_{0,m})-y\|^2/N\leq\ell_0'$~(see \eqref{eq:uniform-init}) and
$|\Delta{\loss_m}|\leq {d(m)}\ell_2'$ on $\overline B_R(w_{0,m})$.
Stopped It\^o's formula gives
\begin{equation}\label{eq:stopped-Ito}
{\loss_m}(w_{T\wedge\tau_R})-{\loss_{m,*}}={\loss_m}(w_{0,m})-{\loss_{m,*}}
-\int_0^{T\wedge\tau_R}\|\nabla{\loss_m}(w_s)\|^2\,ds
+\sigma\int_0^{T\wedge\tau_R}\Delta{\loss_m}(w_s)\,ds+\mathcal M_T,
\end{equation}
where
$\mathcal M_T=\sqrt{2\sigma}\int_0^T
\mathbf1_{\{s<\tau_R\}}\nabla{\loss_m}(w_s)^\top\,dB_s$.
The gradient is bounded on the closed ball, so It\^o's isometry gives
$\E|\mathcal M_T|^2\leq2\sigma T
\sup_{\overline B_R(w_{0,m})}\|\nabla{\loss_m}\|^2<\infty$;
thus $\E\mathcal M_T=0$. Taking expectations in
\eqref{eq:stopped-Ito}, using ${\loss_m}(w_{T\wedge\tau_R})-{\loss_{m,*}}\geq0$
and $T\wedge\tau_R\leq T$,
yields
\begin{equation}\label{eq:energy}
\E\int_0^{T\wedge\tau_R}\|\nabla{\loss_m}(w_s)\|^2\,ds
\leq\ell_0'+\sigma {d(m)}\ell_2'T.
\end{equation}
The integrated SDE and Cauchy--Schwarz imply
\begin{equation}\label{eq:displacement}
\|w_{T\wedge\tau_R}-w_{0,m}\|^2
\leq2T\int_0^{T\wedge\tau_R}\|\nabla{\loss_m}(w_s)\|^2\,ds
   +4\sigma\|B_{T\wedge\tau_R}\|^2.
\end{equation}
Applying the optional stopping theorem to the martingale
$\|B_t\|^2-{d(m)}t$ gives
$\E\|B_{T\wedge\tau_R}\|^2={d(m)}\E(T\wedge\tau_R)\leq {d(m)}T$.
On $\{\tau_R\leq T\}$, continuity gives
$\|w_{T\wedge\tau_R}-w_{0,m}\|=R$.
Combining \eqref{eq:displacement} with \eqref{eq:energy} and the
optional-stopping bound gives
\begin{equation}\label{eq:exit}
R^2\Pp(\tau_R\leq T)
\leq\E\|w_{T\wedge\tau_R}-w_{0,m}\|^2
\leq2T\ell_0'+2\sigma {d(m)}\ell_2'T^2+4{d(m)}\sigma T=C(T).
\end{equation}

For the concentration estimate, the local PL and Laplacian bounds give
$\ell_1'({\loss_m}(w)-{\loss_{m,*}})-\|\nabla{\loss_m}(w)\|^2
+\sigma\Delta{\loss_m}(w)\leq\sigma {d(m)}\ell_2'$ on the prescribed ball.
It\^o's formula gives:
\begin{align*}
&\E\left[e^{\ell_1'(t\wedge\tau_R)}
  ({\loss_m}(w_{t\wedge\tau_R})-{\loss_{m,*}})\right]\\
&\quad={\loss_m}(w_{0,m})-{\loss_{m,*}}
 +\E\int_0^{t\wedge\tau_R}e^{\ell_1's}
 \bigl[\ell_1'({\loss_m}(w_s)-{\loss_{m,*}})-\|\nabla{\loss_m}(w_s)\|^2
      +\sigma\Delta{\loss_m}(w_s)\bigr]\,ds\\
&\quad\leq{\loss_m}(w_{0,m})-{\loss_{m,*}}
   +\sigma {d(m)}\ell_2'\int_0^t e^{\ell_1's}\,ds\\
&\quad\leq\ell_0'+\frac{\sigma {d(m)}\ell_2'}{\ell_1'}(e^{\ell_1't}-1).
\end{align*}
The stochastic integral has zero mean by the same compact-ball
estimate as above, with the bounded factor $e^{\ell_1's}$ for $s\leq t$.
The contribution from $\{\tau_R\leq t\}$ on the left is nonnegative;
discarding it and dividing by $e^{\ell_1't}$ yields
$\E[({\loss_m}(w_t)-{\loss_{m,*}})\mathbf1_{\{\tau_R>t\}}]\leq A(t)$.
Since ${\loss_m}(w_T)-{\loss_{m,*}}\geq0$, Markov's inequality on
$\{\tau_R>T\}$ and the exit bound \eqref{eq:exit} give
\begin{align*}
\Pp\{{\loss_m}(w_T)-{\loss_{m,*}}\geq\epsilon\}
&\leq\Pp\{{\loss_m}(w_T)-{\loss_{m,*}}\geq\epsilon,\ \tau_R>T\}
      +\Pp(\tau_R\leq T)\\
&\leq\frac1\epsilon
 \E\!\left[({\loss_m}(w_T)-{\loss_{m,*}})\mathbf1_{\{\tau_R>T\}}\right]
      +\Pp(\tau_R\leq T)\\
&\leq\frac{A(T)}{\epsilon}+\frac{C(T)}{R^2}.
\end{align*}
This proves \eqref{eq:prob-bound}.
The first condition in \eqref{eq:TR} ensures
$\ell_0' e^{-\ell_1'T}\leq\sigma {d(m)}\ell_2'/\ell_1'$;
the second ensures $C(T)/R^2\leq\delta$.
Thus the probability of ${\loss_m}(w_T)-{\loss_{m,*}}\geq\epsilon$ is at most
$2\sigma {d(m)}\ell_2'/(\ell_1'\epsilon)+\delta$,
proving \eqref{eq:joint-prob}.
\end{proof}

{
\begin{corollary}[Finite-time guarantee across widths]\label{cor:finite-time-family}
Assume the hypotheses of Proposition~\ref{prop:main} and well-posed
Langevin dynamics for each width. For every $\epsilon>0$ and
$0<\eta'<1$, there are $T,R,q>0$ and an integer $m_0$ such that,
for every $m\geq m_0$, the dynamics with noise level $\sigma=q/{d(m)}$ satisfy
$\Pp\{{\loss_m}(w_T)-{\loss_{m,*}}<\epsilon\}\geq1-\eta'$,
where probability is joint over Gaussian initialization and independent
Brownian noise. 
\end{corollary}
\begin{proof}
Choose $\eta,\delta>0$ with $\eta+\delta<\eta'$ and take the
constants in \eqref{eq:constants} for this $\eta$. Set
$q=(\eta'-\eta-\delta)\ell_1'\epsilon/(2\ell_2')$ and
$\sigma=q/{d(m)}$. Then $T_0$ and $C(T)$ in \eqref{eq:TR} are
independent of $m$. Choose $T=T_0$ if $T_0>0$, and any fixed
$T>0$ otherwise. Next choose $R^2\geq C(T)/\delta$ and
$m_0=m_H(R,\eta,N,\kappa)$. Proposition~\ref{prop:main} supplies
the local bounds on this ball with probability at least $1-\eta$;
conditional on that event, Proposition~\ref{prop:noise} gives
\[
\Pp\{{\loss_m}(w_T)-{\loss_{m,*}}<\epsilon\}
\geq(1-\eta)\left(1-\delta
-\frac{2q\ell_2'}{\ell_1'\epsilon}\right)\geq1-\eta'.
\]
\end{proof}
}
}

\section{Well-posedness and regularity}\label{sec:wpreg}
In this section we study the law of the process $\rho_t=\operatorname{Law}(w_t)$, for $w_t$ being the solution of the Langevin dynamics \eqref{eq:langevin}. We recall that $\rho_t$ fulfills the Fokker-Planck equation
\begin{equation}\label{eq:FP}
\partial_t \rho_t(w) = \operatorname{div}(\nabla \obj(w) \rho_t(w)) + \sigma \Delta \rho_t(w).
\end{equation}

We intend to revisit results of well-posedness and regularity of solutions, by introducing novel a priori estimates. These in depth arguments will allow us to characterize the large time behavior, as we do in details in Section \ref{sec:asympt}. We stress here that the results of these last two sections \emph{do not require PL or LSI conditions} and they are somehow independent of the concentration results obtained in previous sections.

Now, let ${\varphi_t(w)}:=\frac{\rho_t(w)}{\pi(w)}$. From \eqref{eq:FP},  it is direct to verify that ${\varphi_t}$ satisfies the following equation
\begin{equation}
    \partial_t{\varphi}_t=\sigma\Delta{\varphi}_t-\inner{\nabla \obj}{\nabla{\varphi}_t}.
\end{equation}
{Define the operator
$\mathfrak{L}{\varphi} := \Delta {\varphi} - \langle \nabla {\varphi}, \nabla \obj \rangle/\sigma$.
Then ${\varphi}_t$ satisfies
\begin{equation}\label{eq:23}
\partial_t {\varphi}_t = {\sigma}\mathfrak{L}{\varphi}_t.
\end{equation}
}
\subsection{Formal a priori estimates and asymptotics}
In this subsection, we provide a priori estimates and asymptotics of  the solution ${\varphi}_t$ of equation \eqref{eq:23} by formal computations, which we will render rigorous in the next subsection. {For convenience of notation, in what follows we use the integration symbol $\int$ to denote integration over $\mathbb{R}^d$.}
\subsubsection{First Order Time Derivative.}
 Using formally integration by parts (which we justify in the proof of Theorem \ref{prop:7} below), we have
\begin{equation}
    \int\inner{\nabla{{\varphi}}}{\nabla\psi}d\pi(w)=-\int{{\varphi}}\mathfrak{L}\psi d\pi(w).
\end{equation}
Thus it is easy to derive the following identity
\begin{equation}
    \frac{d}{dt}\int\norm{\nabla{{\varphi}}_t}^2d\pi(w)=-2{{\sigma}}\int(\mathfrak{L}{{\varphi}}_t)^2d\pi(w)\leq 0.
\end{equation}
Next, we have 
\begin{equation}\label{eq:cntr}
    \frac{d}{dt}\int {{\varphi}}^2_td\pi(w)=-2{{\sigma}}\int\norm{\nabla{{\varphi}}_t}^2d\pi(w)\leq 0,
\end{equation}
thus $\int_0^t\int\norm{\nabla{{\varphi}}_s}^2d\pi(w)ds\leq {1/(2{{\sigma}})}\int{{\varphi}}_0^2d\pi(w),\int{{\varphi}}_t^2d\pi(w)\leq \int{{\varphi}}_0^2d\pi(w)$, and 

\begin{equation}\label{eq:conv2const}
    \int\norm{\nabla{{\varphi}}_T}^2d\pi(w)\leq \frac{1}{T}\int_0^T\int\norm{\nabla{{\varphi}}_t}^2d\pi(w)dt\leq \frac{\int{{\varphi}}_0^2d\pi(w)}{{2{{\sigma}} T}}\to 0,\quad T\to\infty,
\end{equation}
here, we used $\frac{d}{dt}\int\norm{\nabla{{\varphi}}_t}^2d\pi(w)=-2{{\sigma}}\int(\mathfrak{L}{{\varphi}}_t)^2d\pi(w)\leq 0$.
Notice now that \eqref{eq:conv2const} implies that the gradient of ${{\varphi}}_t$ vanishes for $t\to \infty$, meaning that ${{\varphi}}_t$ converges to a constant value on connected sets. The convergence to a constant does not depend on the integrability of $\pi$. 
{Equation \eqref{eq:conv2const} gives $\mathcal O(1/(\sigma T))$ decay of the relative Dirichlet energy of $\varphi_T$. It does not give a rate for convergence of $\varphi_T$ or of the law $\rho_T$. The large-time limits are identified in Section~\ref{sec:asympt}.} 


\subsubsection{Novel higher order a priori estimates}

The following a priori estimates are interesting as they seem not to appear in the broad literature related to the Fokker-Planck equation. They are  useful to obtain in alternative manner well-posedness and regularity of the solution, as we establish in Theorem \ref{prop:7} below. For now, let us collect these estimates for later use as follows.

\paragraph{Second order time derivative.}
We formally compute the time derivative of $\int (\mathfrak{L}{{\varphi}}_t)^2d\pi(w)$.
We have
\begin{equation}
   \begin{aligned}
        \frac{d}{dt}\int (\mathfrak{L}{{\varphi}}_t)^2d\pi(w)&=2\int\mathfrak{L}{{\varphi}}_t\mathfrak{L}(\partial_t{{\varphi}}_t)d\pi(w)\\
        &=2{{\sigma}}\int\mathfrak{L}{{\varphi}}_t\mathfrak{L}(\mathfrak{L}{{\varphi}}_t)d\pi(w)\\
        &=-2{{\sigma}}\int \norm{\nabla\mathfrak{L}{{\varphi}}_t}^2d\pi(w)\leq 0,
   \end{aligned}
\end{equation}
which means $\int (\mathfrak{L}{{\varphi}}_t)^2d\pi(w)$ is monotone non-increasing.
Thus we have
\begin{equation}
   \int(\mathfrak{L}{{\varphi}}_{2T})^2d\pi(w) \leq\frac{1}{T}\int_T^{2T}\int(\mathfrak{L}{{\varphi}}_t)^2d\pi(w)dt\leq \frac{\int\norm{\nabla{{\varphi}}_T}^2d\pi(w)}{2{{\sigma}} T}\leq \frac{\int{{\varphi}}_0^2d\pi(w)}{(2{{\sigma}} T)^2}.
\end{equation}

\paragraph{Third order time derivative.}

Now we calculate the time derivative of $\int \norm{\nabla\mathfrak{L}{{\varphi}}_t}^2d\pi(w)$.
We have
\begin{equation}
   \begin{aligned}
        \frac{d}{dt}\int \norm{\nabla\mathfrak{L}{{\varphi}}_t}^2d\pi(w)&=2\int\inner{\nabla\mathfrak{L}{{\varphi}}_t}{\nabla\mathfrak{L}(\partial_t{{\varphi}}_t)}d\pi(w)\\
        &=-2\int\mathfrak{L}^2{{\varphi}}_t\mathfrak{L}(\partial_t{{\varphi}}_t)d\pi(w)\\
        &=-2{{\sigma}}\int (\mathfrak{L}^2{{\varphi}}_t)^2\leq 0,
   \end{aligned}
\end{equation}
which means $\int \norm{\nabla\mathfrak{L}{{\varphi}}_t}^2d\pi(w)$ is monotone non-increasing.
Thus we have
\begin{equation}
   \int\norm{\nabla\mathfrak{L}{{\varphi}}_{3T}}^2d\pi(w) \leq\frac{1}{T}\int_{2T}^{3T}\int\norm{\nabla\mathfrak{L}{{\varphi}}_t}^2d\pi(w)dt\leq \frac{\int(\mathfrak{L}{{\varphi}}_{2T})^2d\pi(w)}{2{{\sigma}} T}\leq \frac{\int{{\varphi}}_0^2d\pi(w)}{(2{{\sigma}} T)^3}.
\end{equation}

\paragraph{Higher order time derivative.}

Finally, we define the operator 
\begin{equation}
    \mathfrak{F}_k:=\begin{cases}
        \mathfrak{L}^{\frac{k}{2}}& \text{$k$ is even}\\
        \nabla\mathfrak{L}^{\frac{k-1}{2}}& \text{$k$ is odd}.
    \end{cases}
\end{equation}

Then, by induction, we can derive that
\begin{equation}
    \int\norm{\mathfrak{F}_k{{\varphi}}_{kT}}^2d\pi(w)\leq \frac{\int{{\varphi}}_0^2d\pi(w)}{(2{{\sigma}} T)^k},
\end{equation}
or equivalently
\begin{equation}
    \int\norm{\mathfrak{F}_k{{\varphi}}_{T}}^2d\pi(w)\leq \frac{k^k\int{{\varphi}}_0^2d\pi(w)}{(2{{\sigma}})^kT^k}=\Big(\frac{k}{2{{\sigma}} T}\Big)^k\int{{\varphi}}_0^2d\pi(w),\quad\forall T> 0.
\end{equation}
\begin{remark} {\bf Regularity:}
    Given the fact that $\frac{d^k}{dt^k}{{\varphi}}_t={{\sigma}}^k\mathfrak{L}^k{{\varphi}}_t$,
    \begin{equation}
        \int_\tau^t\int_{\mathbb{R}^d}\norm{\mathfrak{L}^k{{\varphi}}_t}^2d\pi(w)dt<\infty, \quad t\geq \tau>0,
    \end{equation}
    and Sobolev embedding, we know that ${{\varphi}}_t$ is automatically smooth both in the space and time variable, provided that $\obj$ is smooth and $\int{{\varphi}}_0^2d\pi(w)<\infty$. 
\end{remark}
\subsection{Rigorous results of well-posedness and regularity}
In the last section, we performed some novel (formal) estimations, by using integration by parts. In this section we show that these arguments are actually legal, by verifying rigorously that ${{\varphi}}_t=\frac{\rho_{{t}}}{\pi}$, where $\rho_t=\operatorname{Law}({w_t})$ does satisfy the estimates under proper (mild) assumptions on $\obj$.

\begin{theorem}\label{prop:7}
    Let $\obj$ be a smooth function and ${{\varphi}}_0 \in L^2(\mathbb{R}^d,\pi)$, ${{\varphi}}_0>0$, then the following equation
    \begin{equation}
            \partial_t{{\varphi}}_t={{{\sigma}}\mathfrak{L}{{\varphi}}_t={{\sigma}}\Delta{{\varphi}}_t-\inner{\nabla \obj}{\nabla{{\varphi}}_t}}
    \end{equation}
    with initial data ${{\varphi}}_0$ has a unique non-negative smooth solution satisfying the following estimate
\begin{equation}
    \int\norm{\mathfrak{F}_k{\varphi}_{t}}^2d\pi(w)\leq \Big(\frac{k}{2{\sigma} t}\Big)^k\int{\varphi}_0^2d\pi(w),\quad\forall t> 0,
\end{equation}
and $2{\sigma}\int_0^t\int\norm{\nabla{\varphi}_t}^2d\pi(w)dt\leq \int {\varphi}_0^2d\pi(w),\int{\varphi}_t^2d\pi(w)\leq \int{\varphi}_0^2d\pi(w)$.
\end{theorem}

\begin{proof} {In the proof, let $\bold{j}=(j_1,\cdots,j_d)$ be a multi-index with non-negative integer components, and denote its order by $|\bold{j}|=\sum_{i=1}^d j_i$. We write ${{\varphi}}^{(\bold{j})}$ for the partial derivative $\frac{\partial^{|\bold{j}|}{{\varphi}}}{\partial w_1^{j_1}\cdots \partial w_d^{j_d}}$. Moreover, $\sum_{|\bold{j}|=k} \|{{\varphi}}^{(\bold{j})}\|$ denotes the sum of the norms of all derivatives of ${{\varphi}}$ of total order $k$.}

   \paragraph{I.} Let $S_i:\mathbb{R}^d\to [0,1]$ be a smooth function  that equals $1$ within $B_{2^i}(0)$ and equals $0$ outside $B_{2^i+1}(0)$. We define $\obj_i(w)=S_i(w)\obj(w)+(1-S_i(w))\norm{w}$, then it is direct to verify that $\norm{\obj_i^{(k)}}$ is bounded on $\mathbb{R}^d$ for any $k\geq 1$ and $\pi_i(w)=e^{-\frac{\obj_i}{{\sigma}}}$ is bounded on $\mathbb{R}^d$. With this $\obj_i$, by \cite[Theorem 2.5]{fornasier2025regularity}, the following Cauchy problem
    \begin{equation}\label{eq:73}
        \begin{cases}
            &\partial_t{{\varphi}}_{i,t}={\sigma}\Delta{{\varphi}}_{i,t}-\inner{\nabla \obj_i}{\nabla{{\varphi}}_{i,t}}=: {{\sigma}}\mathfrak{L}_i{{\varphi}}_{i,t},\\
            &{{\varphi}}_{i,0}(w)=\psi(w)\geq 0,\quad \psi\in C_c^\infty(\mathbb{R}^d)
        \end{cases}
    \end{equation}
    has a unique non-negative solution in $C^\infty([0,T]\times\mathbb{R}^d)\cap W^{1,\infty}(0,T;H^k(\mathbb{R}^d))$, for any $k\geq 0$, the non-negativeness of the solution is by the comparison principle~(since now $\nabla \obj_i$ has bounded derivatives, see \cite[Section 3.1]{vazquez2007porous}). By the properties of $\obj_i$, we have
    \begin{equation}\label{eq:new74}
\begin{alignedat}{2}
\norm{\mathfrak{L}_i^k{{\varphi}}_{i,t}}
 &\leq C_i\sum_{|\bold{j}|=0}^{2k}\norm{{{\varphi}}_{i,t}^{(\bold{j})}},\qquad &
\norm{\nabla\mathfrak{L}_i^k{{\varphi}}_{i,t}}
 &\leq C_i\sum_{|\bold{j}|=0}^{2k+1}\norm{{{\varphi}}_{i,t}^{(\bold{j})}},\\
\norm{\mathfrak{L}_i^k\partial_t{{\varphi}}_{i,t}}
 &\leq C_i\sum_{|\bold{j}|=0}^{2k}\norm{\partial_t{{\varphi}}_{i,t}^{(\bold{j})}},\qquad &
\norm{\nabla\mathfrak{L}_i^k\partial_t{{\varphi}}_{i,t}}
 &\leq C_i\sum_{|\bold{j}|=0}^{2k+1}\norm{\partial_t{{\varphi}}_{i,t}^{(\bold{j})}},
\end{alignedat}
\end{equation}
thus, we have $\nabla\mathfrak{L}_i^k{{\varphi}}_{i,t},\mathfrak{L}_i^k{{\varphi}}_{i,t}\in W^{1,\infty}(0,T;H^m(\mathbb{R}^d))$ for any $k,m>0$.

Now, for any integers $m\geq 1,n\geq 0$,  by integration by parts, we have
\begin{equation}
    \begin{aligned}
        \int_{B_R(0)}\mathfrak{L}_i^m{{\varphi}}_{i,t}\mathfrak{L}_i^n{{\varphi}}_{i,t}d\pi_i(w)&=\int_{\partial B_R(0)}\Big(\frac{\partial}{\partial \nu}\mathfrak{L}_i^{m-1}{{\varphi}}_{i,t}(w)\Big)\mathfrak{L}_i^n{{\varphi}}_{i,t}(w)\pi_i(w)dS\\
        &\quad -\int_{B_R(0)}\inner{\nabla \mathfrak{L}_i^{m-1}{{\varphi}}_{i,t}}{\nabla\mathfrak{L}_i^n{{\varphi}}_{i,t}}d\pi_i(w).
    \end{aligned}
\end{equation}
For the boundary term, we have
{\allowdisplaybreaks[1]
\begin{align}
&\Big|\int_{\partial B_R(0)}\Big(\frac{\partial}{\partial \nu}\mathfrak{L}_i^{m-1}{{\varphi}}_{i,t}\Big)\mathfrak{L}_i^{n}{{\varphi}}_{i,t}(w)\pi_i(w)dS\Big|\notag\\
&\leq C_i\Big(\int_{\partial B_R(0)}\norm{\frac{\partial}{\partial \nu}\mathfrak{L}_i^{m-1}{{\varphi}}_{i,t}}^2dS+\int_{\partial B_R(0)}\norm{\mathfrak{L}_i^{n}{{\varphi}}_{i,t}}^2dS\Big)\notag\\
&\leq C_i\Big(\int_{\partial B_R(0)\cup \partial B_{R+1}(0)}\norm{\frac{\partial}{\partial \nu}\mathfrak{L}_i^{m-1}{{\varphi}}_{i,t}}^2dS+\int_{\partial B_R(0)\cup\partial B_{R+1}(0)}\norm{\mathfrak{L}_i^{n}{{\varphi}}_{i,t}}^2dS\Big)\notag\\
&\leq C_i\sum_{|\bold{j}|=0}^{\max\{2m-1,2n\}}\int_{\partial B_R(0)}\norm{{{\varphi}}_{i,t}^{(\bold{j})}}^2dS\notag\\*
&\quad +C_i\sum_{|\bold{j}|=0}^{\max\{2m-1,2n\}}\int_{\partial B_{R+1}(0)}\norm{{{\varphi}}_{i,t}^{(\bold{j})}}^2dS\notag\\
&\leq C_i\sum_{|\bold{j}|=0}^{\max\{2m,2n+1\}}\int_{B_{R+1}(0)\setminus B_R(0)}\norm{{{\varphi}}_{i,t}^{(\bold{j})}}^2dw\notag\\
&\to 0,\quad\text{ as }R\to\infty,
\end{align}
}
     in the above, the first inequality is due to Cauchy-Schwartz inequality, the third one is due to the estimates \eqref{eq:new74}, and the last one is due to the trace theorem. We note that the constant in the trace theorem here is independent of $R$. This follows by decomposing $B_{R+1}(0) \setminus B_R(0)$ into a disjoint union of smaller regions, applying the trace theorem on each, and summing the results (see the proof of \cite[Lemma 2.13]{fornasier2025regularity} for details).
    
    Thus, let $R\to\infty$, we have
        \begin{equation}
\int\mathfrak{L}_i^m{{\varphi}}_{i,t}\mathfrak{L}_i^n{{\varphi}}_{i,t}d\pi_i(w)=-\int\inner{\nabla \mathfrak{L}_i^{m-1}{{\varphi}}_{i,t}}{\nabla\mathfrak{L}_i^n{{\varphi}}_{i,t}}d\pi_i(w),
        \end{equation}
for any integers $m\geq 1,n\geq 0$, which means that the use of integration by parts in the last section is legal for ${{\varphi}}_{i,t}$. Thus we proved that the following estimates hold
\begin{equation}
    \int\norm{\mathfrak{F}_{k,i}{{\varphi}}_{i,t}}^2d\pi_i(w)\leq \frac{k^k\int\psi^2d\pi_i(w)}{(2{{\sigma}})^kt^k}=\Big(\frac{k}{2{{\sigma}} t}\Big)^k\int\psi^2d\pi_i(w),\quad\forall t> 0,
\end{equation}
and $2{{\sigma}}\int_0^t\int\norm{\nabla{{\varphi}}_{i,s}}d\pi_i(w)ds\leq \int \psi^2d\pi_i(w),\int ({{\varphi}}_{i,t})^2d\pi_i(w)\leq \int \psi^2d\pi_i(w)$. Recall that $\psi$ is the initial non-negative datum in \eqref{eq:73}.

\paragraph{II.}By interior regularity theory of elliptic equation~(see Gårding's inequality \cite[Theorem 3.54]{aubin2012nonlinear}), if $\mathfrak{L}^k_ig=f$ on $U$ weakly, then for $V\Subset U$, we have
\begin{equation}
    \norm{g}_{H^{2k}(V)}\leq C_{i,k}(V,U)\Big(\norm{f}_{L^2(U)}+\norm{g}_{L^2(U)}\Big)=C_{i,k}(V,U)\Big(\norm{\mathfrak{L}_i^k g}_{L^2(U)}+\norm{g}_{L^2(U)}\Big),
\end{equation}
thus, for $h\geq i$, we have
\begin{equation}
    \sum_{|\bold{j}|=0}^{2m}\int_{B_i(0)}\norm{{{\varphi}}_{h,t}^{(\bold{j})}}^2dw\leq C_{m,i}\sum_{j=0}^{2m}\int_{B_i(0)}\norm{\mathfrak{F}_{j,h}{{\varphi}}_{h,t}}^2d\pi_i(w)\leq C_{m,i,t,\psi}<\infty,
\end{equation}
here we denote $B_i(0):=B_{2^i}(0)$ for simplicity.
By Sobolev embedding,
\[
\norm{{{\varphi}}_{h,t}}_{L^\infty(B_i(0)\times[t_0,\infty))}
\leq C_{m,i,t_0,\psi}<\infty.
\]
Thus by Shauder's interior estimate of parabolic equation~(see for example \cite[Theorem 18]{sun2025well}), we have
\begin{equation}\label{eq:81}
    \norm{{{\varphi}}^{(k)}_{h,t}}_{2+\delta,1+\delta/2; B_{i-1}(0)\times [t_0,t_1]}\leq C_{m,k,i,t_0,t_1,\psi}<\infty,
\end{equation}
for some $\delta\in (0,1)$ and any $k\geq 0$. 

{

Now, by a diagonal selection procedure, we can find a smooth non-negative limit function satisfying the equation 
\begin{equation}\label{eq:82}
    \begin{cases}
        &\partial_t{{\varphi}}_t={\sigma}\Delta{{\varphi}}_t-\inner{\nabla \obj}{\nabla{{\varphi}}_t},\\
        &{{\varphi}}_0(w)=\psi(w),
    \end{cases}
\end{equation}
weakly. As this function is smooth, it also satisfies the equation classically.
Fatou's lemma gives, for every $t>0$,
\[
\begin{gathered}
\int\norm{\mathfrak{F}_{k}{{\varphi}}_{t}}^2d\pi(w)
\leq \Big(\frac{k}{2{{\sigma}}t}\Big)^k\int\psi^2d\pi(w),\\
2{{\sigma}}\int_0^t\int\norm{\nabla{{\varphi}}_{s}}d\pi(w)ds
\leq\int\psi^2d\pi(w),\qquad
\int({{\varphi}}_{t})^2d\pi(w)\leq\int\psi^2d\pi(w).
\end{gathered}
\] 

For reader's convenience, we add more detail for this diagonal selection procedure~(we fix $k=10$, which is large enough for us to derive a limit function).  We will use estimate \eqref{eq:81}, in which the upper bound does not depend on $h$, and the fact that the unit ball in Banach space $C^{k+2+\delta,1+\delta/2}(B_R(0)\times [t_0,t_1])$ is compact in $C^{k+2+\delta-\epsilon,1+\delta/2-\epsilon}(B_R(0)\times [t_0,t_1])$ for any small $\epsilon>0$. 
Now for $B_1(0)\times [2^{-1},2^1]$, we can extract a diverging subsequence $(h_{1j})_{j=1}^\infty$ of positive integers, such that ${{\varphi}}_{h_{1j},t}\to {{\varphi}}_t$ in $C^{12+\delta-\epsilon,1+\delta/2-\epsilon}(B_1\times [2^{-1},2])$, here we denote ${{\varphi}}_t$ as the limit function in 
$C^{12+\delta,1+\delta/2}(B_1\times [2^{-1},2])$, and on $B_1(0)\times [2^{-1},2^1]$, we have $\partial_t{{\varphi}}_t={\sigma}\Delta{{\varphi}}_t-\inner{\nabla \obj}{\nabla{{\varphi}}_t}$ point-wise, since $\partial_t{{\varphi}}_{h_{1j},t}={\sigma}\Delta{{\varphi}}_{h_{1j},t}-\inner{\nabla \obj}{\nabla{{\varphi}}_{h_{1j},t}}$ point-wise, for $j$ large enough.
Now, on $B_{2}(0)\times [2^{-2},2^2]$, we can extract a subsequence $(h_{2j})_{j=1}^\infty$ from $(h_{1j})_{j=1}^\infty$, such that the limit function~(which is still denoted as ${{\varphi}}_t$ and equals the previous limit function on $B_1\times [2^{-1},2]$)  ${{\varphi}}_t\in C^{12+\delta,1+\delta/2}(B_2\times [2^{-2},2^2])$ and 
satisfies point-wise the equation on $B_2\times [2^{-2},2^2]$. We can keep this selection 
procedure by induction, that is extracting a  subsequence $(h_{ij})_{j=1}^\infty$ from $(h_{(i-1)j})_{j=1}^\infty$, and the limit function again denoted as ${{\varphi}}_t\in C^{12+\delta,1+\delta/2}(B_i\times [2^{-i},2^i])$  which satisfies the equation point-wise on $B_i\times [2^{-i},2^i]$.
Now, we let $h_i:=h_{ii}$, then ${{\varphi}}_{h_i,t}$ will converge to ${{\varphi}}_t$ in $C_{loc}^{12+\delta-\epsilon,1+\delta/2-\epsilon}(\mathbb{R}^d\times (0,\infty))$, and this ${{\varphi}}_t$ satisfies the equation on $\mathbb{R}^d\times (0,\infty)$ point-wise.
Next, we verify ${{\varphi}}_t$ is a weak solution to the Cauchy problem \eqref{eq:82}~(hence ${{\varphi}}_t$ is also a classical solution to the Cauchy problem \eqref{eq:82}). Let $f\in C_c^\infty(\mathbb{R}^d)$ be a test function, whose support we denote as $\Omega$. Then for each $i$, we have
\begin{equation}\label{eq:84}
    \begin{aligned}
        \int_\Omega {{\varphi}}_{h_{i},T}(w)f(w)d\pi_{h_{i}}(w)-\int_\Omega \psi(w)f(w)d\pi_{h_{i}}(w)=-{{\sigma}}\int_{0}^T\int_\Omega\inner{\nabla{{\varphi}}_{h_{i},t}(w)}{\nabla f(w)}d\pi_{h_{i}}(w)dt,
    \end{aligned}
\end{equation}
for all $T>0$, since ${{\varphi}}_{h_{i},t}$ is a classical solution to the Cauchy problem \eqref{eq:73}. Now, since ${{\varphi}}_{h_{i},T}\to{{\varphi}}_T$ in $C_{loc}^{12+\delta-\epsilon,1+\delta/2-\epsilon}(\mathbb{R}^d\times (0,\infty))$ and $\pi_{h_i}\to \pi$ in $C^k_{loc}(\mathbb{R}^d)$ for any $k\geq 1$, {the left-hand side of \eqref{eq:84} converges to
$\int_\Omega {{\varphi}}_{T}(w)f(w)d\pi(w)-\int_\Omega \psi(w)f(w)d\pi(w)$.}
For the right hand side, we have $\norm{\nabla{{\varphi}}_{h_{i},t}}_{ L^{2}(\Omega\times [0,T])}\leq C(\psi,\Omega,\mathcal{L})<\infty$ for $i$ large enough (which is a direct consequence of estimate \eqref{eq:conv2const}). Thus in the space $L^{2}(\Omega\times [0,T])$ there is a weak limit of $\nabla{{\varphi}}_{h_{i},t}$, which we denote as $\tilde{\Phi}\in L^{2}(\Omega\times [0,T])$, that is 
\begin{eqnarray}  \lim_{i\to\infty}\int_{0}^T\int_{\Omega}\inner{\nabla{{\varphi}}_{h_{i},t}(w)}{\nabla f(w)}d\pi_{h_{i}}(w)dt=\int_{0}^T\int_{\Omega}\inner{\tilde{\Phi}(w)}{\nabla f(w)}d\pi(w)dt.
\end{eqnarray}
On the other hand, we have $\nabla{{\varphi}}_{h_{i},t}\to\nabla{{\varphi}}_t$ in $C_{loc}^{11+\delta-\epsilon,1+\delta/2-\epsilon}(\mathbb{R}^d\times (0,\infty))$, thus we have $\tilde{\Phi}=\nabla{{\varphi}}_t$, and   ${{\varphi}}_t$ satisfies Cauchy problem \eqref{eq:82} weakly. 

In the above derivation, we only fixed $k=10$ and derived ${{\varphi}}_t\in C_{loc}^{12+\delta,1+\delta/2}(\mathbb{R}^d\times (0,\infty))$, but for general $k$ we can use 
the estimate \eqref{eq:81} to do the same derivation to get this ${{\varphi}}_t\in C_{loc}^{k+2+\delta,1+\delta/2}(\mathbb{R}^d\times (0,\infty))$. For example, to get the result for $k=20$, in the first step of the above diagonal selection procedure, we extract a new subsequence $(h_{1j})_{j=1}^\infty$ from $(h_i)_{i=1}^\infty$~(we defined this sequence above for $k=10$) instead of the original positive integers, which can guarantee the limit function equals ${{\varphi}}_t$ defined for $k=10$; the rest of the induction steps are the same as the case for $k=10$; then we get ${{\varphi}}_t\in C_{loc}^{20+2+\delta,1+\delta/2}(\mathbb{R}^d\times (0,\infty))$. Since $k$ can be any integer, we get ${{\varphi}}_t\in C_{loc}^{k+2+\delta,1+\delta/2}(\mathbb{R}^d\times (0,\infty))$ for any $k$ integer and for each $k$, there is a subsequence $(\tilde{h}_{kj})_{j=1}^\infty$, such that ${{\varphi}}_{\tilde{h}_{kj},t}\to {{\varphi}}_t$ in $C_{loc}^{k+2+\delta-\epsilon,1+\delta/2-\epsilon}(\mathbb{R}^d\times (0,\infty))$.

}

\paragraph{III.}Since the right hand side of the above estimates only depends 
on the $L^2(\mathbb{R}^d,\pi)$ norm of the initial data $\psi$, 
so for any initial data ${{\varphi}}_0\in L^2(\mathbb{R}^d,\pi)$, 
we can use a sequence of $C_c^\infty$ functions $\psi_k$ to 
approximate ${{\varphi}}_0$ in space $L^2(\mathbb{R}^d,\pi)$. Then following a similar approximation procedure as in step \textbf{II}, we can find a limit non-negative function that solves the Cauchy problem with initial data ${{\varphi}}_0$  satisfying the estimate
\begin{equation}
    \int\norm{\mathfrak{F}_{k}{{\varphi}}_{t}}^2d\pi(w)\leq \frac{k^k\int{{\varphi}}_0^2d\pi(w)}{(2{{\sigma}})^kt^k}=\Big(\frac{k}{2{{\sigma}} t}\Big)^k\int{{\varphi}}_0^2d\pi(w),\quad\forall t> 0,
\end{equation}
and ${{\sigma}}\int_0^t\int\norm{\nabla{{\varphi}}_{s}}d\pi(w)ds\leq \int {{\varphi}}_0^2d\pi(w),\int ({{\varphi}}_{t})^2d\pi(w)\leq \int {{\varphi}}_0^2d\pi(w)$.

\paragraph{IV.} As to the uniqueness of such kind of solution, we can prove in the following:
let $\tilde{{\varphi}}_t={{\varphi}}_{1,t}-{{\varphi}}_{2,t}$, here ${{\varphi}}_{i,t}$ is the solution to the Cauchy problem with the same initial data ${{\varphi}}_0$ and satisfies the above estimates. Then we know that $\tilde{{\varphi}}$ is the solution to the Cauchy problem with initial data $0$, and $\tilde{{\varphi}}_t\in L^\infty(0,T;L^2(\mathbb{R}^d,\pi))\cap L^2(0,T;H^1(\mathbb{R}^d,\pi))$.

We have
\begin{equation}
    \begin{aligned}
        \int(\tilde{{\varphi}}_t)^2d\pi(w)-\int(\tilde{{\varphi}}_\tau)^2d\pi(w)&=-2{{\sigma}}\int_\tau^t\int\inner{\nabla\tilde{{\varphi}}_s}{\nabla\tilde{{\varphi}}_s}d\pi(w)ds\\
      &\leq 0,
    \end{aligned}
\end{equation}
{Letting $\tau=0$ gives $\int(\tilde{{\varphi}}_t)^2d\pi(w)\leq 0$,
so $\tilde{{\varphi}}_t\equiv 0$.} Thus the solution is unique.
\end{proof}
\\

In {Theorem} \ref{prop:7}, uniqueness is established in the space $ L^\infty(0,T;L^2(\mathbb{R}^d,\pi))\cap L^2(0,T;H^1(\mathbb{R}^d,\pi))$. However, it is unclear whether $\frac{\rho_{{t}}}{\pi}$ belongs to this space or not, where  $\rho_t=\operatorname{Law}(w_t)$. In the next, we will show that $\frac{\rho_{{t}}}{\pi}$ is a solution to equation \eqref{eq:23} in the sense of Definition \ref{def:8}  and a uniqueness result is obtained in such spaces. Moreover, by establishing uniqueness in the relevant function spaces, we confirm that the estimates from {Theorem} \ref{prop:7} also apply to $\frac{\rho_{{t}}}{\pi}$. The following definitions are adapted from \cite[Chapter 9]{bogachev2022fokker}, here we only care about the non-negative solution.
\begin{definition}\label{def:8}
Fix $T>0$ and ${\varphi}_t$ be the non-negative solution to equation \eqref{eq:23} in the distributional sense with initial data ${\varphi}_0$. For any set $U\subset\mathbb{R}^d$, we denote ${\varphi}_t\pi(U):=\int_{U}{\varphi}_t(w)\pi(w)dw$.
    \begin{itemize}
        \item Subprobability solution: $\mathcal{SP}_{{\varphi}_0}:=\{{\varphi}_t: {\varphi}_t\pi(\mathbb{R}^d)\leq {\varphi}_0\pi(\mathbb{R}^d),\text{ for almost every }t\in (0,T),\\
        \int_0^T\int_U\norm{\nabla \obj(w)}^2{\varphi}_t\pi(w)dwdt<\infty, \text{ for every ball } U\subset\mathbb{R}^d\}$;
        \item Integrable solution: $\mathcal{I}_{{\varphi}_0}=:\{{\varphi}_t:\sup_{t\in (0,T)}{\varphi}_t\pi(\mathbb{R}^d)<\infty, \int_0^T\int_U\norm{\nabla \obj}^p{\varphi}_t\pi(w)dw<\infty,\\
        \text{ for every ball }U\subset\mathbb{R}^d,\text{ and some } p>d+2\}$.
    \end{itemize}
\end{definition}

Obviously, the solution derived in {Theorem} \ref{prop:7} is in class $\mathcal{SP}_{{\varphi}_0}\cap\mathcal{I}_{{\varphi}_0}$.

By \cite[Theorem 9.4.5, Theorem 9.6.3]{bogachev2022fokker}, if $\mathfrak{L}$ satisfies condition A, then $\mathcal{SP}_{{\varphi}_0}$ contains at most one solution of \eqref{eq:23} in the distributional sense; if $\mathfrak{L}$ satisfies condition B, then $\mathcal{I}_{{\varphi}_0}$ contains at most one solution of \eqref{eq:23} in the distributional sense.
\begin{itemize}
\item Condition A. There exist a positive $C^2$ function $V$, such that $V(w)\to\infty$ as $\norm{w}\to\infty$, and a positive constant $C$, such that
\begin{equation}\label{eq:a36}
    \mathfrak{L}V(w)\leq C+CV(w);
\end{equation}
\item Condition B. There exist a positive $C^2$ function $V$, such that $V(w)\to\infty$ as $\norm{w}\to\infty$, and a positive constant $C$,  such that
\begin{equation}\label{eq:a37}
    \mathfrak{L}V(w)\geq -C-CV(w),\quad \norm{\nabla V(w)}\leq C+CV(w).
\end{equation}
\end{itemize}

\begin{example}
Let $V(w)=\log(1+\norm{w}^2)$, then condition \eqref{eq:a36} is satisfied when
\begin{equation}
    {-\inner{\nabla \obj(w)}{w}\leq C\Big(1+\norm{w}^2\log(1+\norm{w}^2)\Big).}
\end{equation}

Let $V(w)=\log(\log(e+\norm{w}))$, then condition \eqref{eq:a37} is satisfied when
\begin{equation}
    {\inner{\nabla \obj(w)}{w}\leq C\norm{w}^2 \log(\norm{w}) \log (\log (1+\|w\|)),}
\end{equation}
for $w$ such that $\norm{w}$ large enough.
\end{example}

\begin{proposition}
    Let $\rho_t$ be the unique distribution of the Langevin dynamic \eqref{eq:langevin} with initial distribution ${\varphi}_0\pi$,here ${\varphi}_0\in L^2(\mathbb{R}^d,\pi)$. If operator $\mathfrak{L}$ satisfies condition A or condition B, then we have $\rho_t$ is smooth for $t>0$, and ${\varphi}'_t:=\frac{\rho_{{t}}}{\pi}={\varphi}_t$, here ${\varphi}_t$ is from {Theorem} \ref{prop:7}.
\end{proposition}
\begin{proof}
    By \cite[Theorem 2.1]{de2011smoothness}, we know the law of the Langevin dynamic~\eqref{eq:langevin} has a smooth density $\rho_t,\forall t>0$~(since $\nabla \obj$ is local Lipschitz continuous, so the distribution is unique), which means ${\varphi}'_t$ is smooth in the space variable, thus we have ${\varphi}_t'\in \mathcal{SP}_{{\varphi}_0}\cap\mathcal{I}_{{\varphi}_0}$. If $\mathfrak{L}$ satisfies condition A or condition B, then we have ${\varphi}_t'={\varphi}_t$, since ${\varphi}_t\in \mathcal{SP}_{{\varphi}_0}\cap\mathcal{I}_{{\varphi}_0}$.
\end{proof}

{
\section{Characterizing the large time behavior} \label{sec:asympt}
In this section, under the assumptions in Theorem \ref{thm:main}, we wish to characterize, again independently of the integrability of $\pi(w)=e^{-\frac{\obj(w)}{{\sigma}}}$,   the large-time behavior of the law of the process $\rho_t=\operatorname{Law}(w_t)$~(equivalently the large-time behavior of $\varphi_t=\rho_t/\pi$), for $w_t$ being the solution of the Langevin dynamics \eqref{eq:langevin}. 


Consider the symmetric non-negative Dirichlet form
\[
\mathcal E(f,g)
:=
\int_{\mathbb R^d}
\langle \nabla f,\nabla g\rangle\,d\pi,
\qquad
f,g\in C_c^\infty(\mathbb R^d).
\]
Let $(\mathcal E,D(\mathcal E))$ denote its closure in
$L^2(\mathbb R^d,\pi)$~(actually, here we have $D(\mathcal{E})=H^1(\mathbb{R}^d,\pi)$, this is due to $\obj$ is lower bounded and so $H^1(\mathbb{R}^d,\pi)$ is exactly the closure of $C_c^\infty(\mathbb{R}^d)$ under $\mathcal{E}_1$-norm, where $\norm{f}^2_{\mathcal{E}_1}:=\normsq{f}_{L^2(\mathbb{R}^d,\pi)}+\mathcal{E}(f,f)$). Since
\[
\mathcal E(f,g)
=
-\langle \mathfrak L f,g\rangle_{L^2(\R,\pi)},
\qquad
f,g\in C_c^\infty(\mathbb R^d),
\]
it follows from the representation theorem for closed non-negative
forms that the operator $\mathfrak L$ admits a unique self-adjoint
extension (still denoted by $\mathfrak L$) satisfying
\[
\mathcal E(f,g)
=
-\langle \mathfrak L f,g\rangle_{L^2(\R,\pi)},
\qquad
f\in D(\mathfrak L),\quad g\in D(\mathcal E).
\]

Since $\mathcal{E}_1(f,g) := \mathcal{E}(f,g) + \langle f, g \rangle_{L^2(\R,\pi)}$ 
induces a norm equivalent to the standard norm of the Hilbert space $H^1(\R,\pi)$, 
the pair $(\mathcal{E}, D(\mathcal{E}))$ forms a closed symmetric Dirichlet form on 
$L^2(\R,\pi)$. By the representation theorem for Dirichlet forms, the associated 
operator $\mathfrak{L}$ is self-adjoint and generates a 
strongly continuous symmetric contraction semigroup $ e^{\sigma t\mathfrak{L}}$ on $L^2(\R,\pi)$. Hence the unique solution of
$\partial_t\varphi_t=\sigma\mathfrak L\varphi_t$
is given by
\begin{equation}
\label{eq:semigroup-representation}
\varphi_t=e^{\sigma t\mathfrak L}\varphi_0.
\end{equation}

For further details on symmetric Dirichlet forms, we refer the reader to
\cite{FukushimaOshimaTakeda2010}.

\begin{theorem}[Large-time dichotomy]\label{thm:large-time-dichotomy}
Let the assumptions in Theorem \ref{thm:main} hold and denote the total mass by
\[
    M_\pi:=\pi(\mathbb{R}^d)
    =\int_{\mathbb{R}^d}e^{-\obj(w)/{\sigma}}\,dw\in(0,\infty]
\]
Then the following assertions hold.
\begin{enumerate}
    \item If \(M_\pi<\infty\), then
    \begin{equation}\label{eq:finite-pi-L2-limit}
        \varphi_t\longrightarrow \frac{1}{M_\pi}
        \qquad\text{strongly in }L^2(\mathbb{R}^d,\pi).
    \end{equation}
    Consequently,
    \begin{equation}\label{eq:finite-pi-TV-limit}
        \int_{\mathbb{R}^d}
        \left|\varphi_t(w)-\frac{1}{M_\pi}\right|d\pi(w)
        \longrightarrow0,
    \end{equation}
    and hence \(\rho_t\) converges in total variation to the normalized Gibbs
    probability measure \(M_\pi^{-1}\pi\).

    \item If \(M_\pi=\infty\), then
    \begin{equation}\label{eq:infinite-pi-L2-limit}
        \varphi_t\longrightarrow0
        \qquad\text{strongly in }L^2(\mathbb{R}^d,\pi).
    \end{equation}
    Consequently, \(\rho_t\) converges vaguely to the zero measure: for every
    \(f\in C_c(\mathbb{R}^d)\),
    \begin{equation}\label{eq:vague-limit}
        \int_{\mathbb{R}^d}f(w)\,\rho_t(dw)\longrightarrow0.
    \end{equation}
    In particular, for every compact set \(Z\Subset\mathbb{R}^d\),
    \begin{equation}\label{eq:compact-escape}
        \rho_t(Z)\longrightarrow0.
    \end{equation}
\end{enumerate}
\end{theorem}

\begin{proof}These results and their proofs are standard, following an established framework well-documented in the functional analysis literature.

Let \(\mathsf E_{\mathfrak L}\) denote the spectral resolution of
\(\mathfrak L\), and let
\(\Pi_0:=\mathsf E_{\mathfrak L}(\{0\})\)
be the orthogonal projection onto \(\ker\mathfrak L\).
By the spectral theorem and the functional calculus for non-negative
self-adjoint operators together with
\eqref{eq:semigroup-representation}, we have
\begin{equation}\label{eq:spectral-convergence}
\begin{aligned}
    \norm{\varphi_t-\Pi_0\varphi_0}_{L^2(\R,\pi)}^2
    &=
    \int_{(0,\infty)}
    e^{-2\sigma t\lambda}\,
    d\nu_{\varphi_0}(\lambda)\\
    &\longrightarrow 0,
\end{aligned}
\end{equation}
where
\(\nu_{\varphi_0}(B):=
\norm{\mathsf E_{\mathfrak L}(B)\varphi_0}_{L^2(\R,\pi)}^2\).
Indeed, the integrand converges pointwise to zero on
\((0,\infty)\), is bounded by \(1\), and
\(\nu_{\varphi_0}((0,\infty))
\leq\norm{\varphi_0}_{L^2(\R,\pi)}^2<\infty\);
hence $1$ is $\nu_{\varphi_0}$ integrable and the conclusion follows from the dominated convergence theorem. See
\cite[Proposition~24]{DunfordSchwartz1963}
and, alternatively,
\cite[Corollary~7.1]{Markin2018}.

It remains to identify \(\ker\mathfrak L\). If \(u\in\ker\mathfrak L\), then
\[
    0=\inner{\mathfrak L u}{u}_{L^2(\R,\pi)}
    =\mathcal E(u,u)
    =\int_{\mathbb{R}^d}\norm{\nabla u}^2d\pi.
\]
Since \(\pi\) is strictly positive and \(\mathbb{R}^d\) is connected, this
implies that \(u\) is constant almost everywhere. If \(M_\pi<\infty\), we have \(1\in D(\mathcal E)\), and
\(\mathcal E(1,g)=0\) for every \(g\in D(\mathcal E)\); hence
\(1\in\ker\mathfrak L\). If $M_\pi=\infty$, then any non-zero constant will not be in $L^2(\mathbb{R}^d,\pi)$, thus $u=0$ on $\mathbb{R}^d$. Therefore
\[
    \ker\mathfrak L=
    \begin{cases}
        \operatorname{span}\{1\},&M_\pi<\infty,\\
        \{0\},&M_\pi=\infty.
    \end{cases}
\]
If \(M_\pi<\infty\), we have
\(\int\varphi_0d\pi=\int \rho_0(dw)=1\), and hence
\[
    \Pi_0\varphi_0
    =\frac{\inner{\varphi_0}{1}_{L^2(\R,\pi)}}{\norm{1}_{L^2(\R,\pi)}^2}
    =\frac{1}{M_\pi}.
\]
Combining this with \eqref{eq:spectral-convergence} proves
\eqref{eq:finite-pi-L2-limit}. Moreover, by Cauchy--Schwarz,
\[
    \int\left|\varphi_t-\frac{1}{M_\pi}\right|d\pi
    \leq M_\pi^{1/2}
    \norm{\varphi_t-\frac{1}{M_\pi}}_{L^2(\R,\pi)},
\]
which proves \eqref{eq:finite-pi-TV-limit}.

If \(M_\pi=\infty\), then \(\Pi_0=0\), so
\eqref{eq:spectral-convergence} gives \eqref{eq:infinite-pi-L2-limit}. For
\(f\in C_c(\mathbb{R}^d)\),
\[
    \left|\int f\,d\rho_t\right|
    =\left|\int f\varphi_t\,d\pi\right|
    \leq \norm{f}_{L^2(\pi)}\norm{\varphi_t}_{L^2(\R,\pi)}
    \longrightarrow0,
\]
which proves vague convergence. Since \(\pi(Z)<\infty\) for every compact
\(Z\),
\[
    \rho_t(Z)
    =\int_Z\varphi_t\,d\pi
    \leq \pi(Z)^{1/2}\norm{\varphi_t}_{L^2(\R,\pi)}
    \longrightarrow0.
\]
\end{proof}

\begin{remark}[Meaning of the \(1/(\sigma t)\) estimate]
For every bounded connected Lipschitz domain
\(Z\Subset\mathbb{R}^d\), define the weighted average
\(m_Z(t):=\pi(Z)^{-1}\int_Z\varphi_t\,d\pi\).
Because \(\pi\) is smooth and strictly positive, it is bounded above and
below on \( Z\). Hence the weighted Poincar\'e inequality on \(Z\)
and \eqref{eq:slowconv} give
\begin{equation}\label{eq:local-flatness-rate}
    \int_Z\left|\varphi_t-m_Z(t)\right|^2d\pi
    \leq \frac{C_Z}{2\sigma t}
    \int_{\mathbb{R}^d}\varphi_0^2d\pi.
\end{equation}
Thus the squared local oscillation decays at rate
\(\mathcal O((\sigma t)^{-1})\). Without an additional spectral or
functional inequality, however, this estimate does not control the rate at
which \(m_Z(t)\) approaches its limiting value, and therefore does not yield
a quantitative rate for the convergence of the law in
Theorem~\ref{thm:large-time-dichotomy}. 
\end{remark}

\begin{remark}
Let \(C_\sigma:=8\sigma\ell_3/[(\ell_1-\sigma\ell_2)\ell_1]\).
Proposition~\ref{prop:concentration} states that, for all sufficiently large
\(t\),
\[
    \rho_t(M_\varepsilon)\geq1-\frac{C_\sigma}{\varepsilon^2},
    \qquad
    M_\varepsilon:=\{w:\dist{w}\leq \varepsilon\}.
\]
If \(M_\pi=\infty\), Theorem~\ref{thm:large-time-dichotomy} implies that
\(\rho_t(Z)\to0\) for every compact \(Z\). Therefore, for every
\(r>\sqrt{C_\sigma}\), the set \(M_\varepsilon\) cannot be compact. In particular,
under the global Polyak--{\L}ojasiewicz condition, the set of minimizers
\(\mathcal W^*\) cannot be compact when \(\pi\) is non-integrable.

Indeed, the quadratic-growth consequence of the global PL condition gives
\begin{equation}\label{eq:QG}
    \obj(w)-\obj_*\geq\kappa\dist{w}^2,
    \qquad w\in\mathbb{R}^d,
\end{equation}
for some \(\kappa>0\). If \(\mathcal W^*\) is compact, then
\(\dist{w}\geq(\norm{w}-R)_+\) for some \(R>0\), and
\eqref{eq:QG} implies that \(e^{-\obj/\sigma}\) has a Gaussian upper bound at
infinity; hence \(M_\pi<\infty\).

Conversely, suppose that \(\mathcal W^*\) is unbounded and that
\(\obj(w)-\obj_*\leq H(\dist{w})\)
for some non-negative continuous function \(H\). Fix \(r>0\), and choose a
sequence \((w_n)\subset\mathcal W^*\) such that the balls \(B_r(w_n)\) are
pairwise disjoint. On each such ball,
\(\obj(w)-\obj_*\leq \max_{0\leq s\leq r}H(s)=:C_r\).
Consequently,
\[
    M_\pi\geq\sum_{n=1}^{\infty}
    \int_{B_r(w_n)}e^{-\obj(w)/\sigma}\,dw
    \geq\sum_{n=1}^{\infty}
    |B_r|e^{-(\obj_*+C_r)/\sigma}=\infty.
\]
\end{remark}

In summary, when \(M_\pi<\infty\), the law converges in total variation to
the normalized Gibbs measure. When \(M_\pi=\infty\), mass escapes every
compact set and \(\rho_t\) converges vaguely to zero. Under the assumptions
of this paper, the explicit \(\mathcal O((\sigma t)^{-1})\) estimate applies
to the Dirichlet energy and to the squared local oscillation in
\eqref{eq:local-flatness-rate}, but not in general to the convergence of the
law itself.

For related results concerning the large-time behavior of Stochastic Gradient Descent, we refer the reader to~\cite{arora22,shalova2024}. In these works, under smoothness assumptions on the manifold $\mathcal{W}^*$, the authors demonstrate that the dynamics exhibit a random walk around  the set $\mathcal{W}^*$ in the large-time limit.
}

\section*{Acknowledgement}
Massimo Fornasier and Lukang Sun acknowledge the support from the Munich Center for Machine Learning. Massimo Fornasier acknowledges the support of the European Research Council (ERC) under the European Union’s Horizon Europe research and innovation programme (grant agreement No. 101198055, project acronym NEITALG).
Rachel Ward acknowledges the support from the Humboldt Foundation Research Award, as well as AFOSR Grant No. FA9550-19-1-0005, NSF Grant 2019844, and NSF Grant 2217069. 

\bigskip
\begin{center}
  \FundingLogos
  
  \vspace{0.5em}
  \begin{tcolorbox}\centering\small
   
    Funded by the European Union. Views and opinions expressed are however those of the author(s) only and do not necessarily reflect those of the European Union or the European Research Council Executive Agency. Neither the European Union nor the granting authority can be held responsible for them. This project has received funding from the European Research Council (ERC) under the European Union’s Horizon Europe research and innovation programme (grant agreement No. 101198055, project acronym NEITALG).
    
  \end{tcolorbox}
\end{center}

\bibliography{bib.bib}
\bibliographystyle{unsrt}

\end{document}